\documentclass[10pt]{amsart}
\usepackage{amsrefs}
\usepackage{amssymb}
\usepackage[all,arc]{xy}
\usepackage{graphicx}
\usepackage{mathrsfs}
\usepackage{xcolor}

\usepackage{hyperref}

\DeclareMathOperator{\Aut}{Aut}

\DeclareMathOperator{\image}{Im}

\DeclareMathOperator{\kernel}{ker}

\DeclareMathOperator{\Res}{res}

\newcommand{\beue}{\bar{\eue}}
\newcommand{\bone}{\mathbf{1}}
\newcommand{\gr}{\mathrm{gr}}
\newcommand{\argu}{\hbox to 7truept{\hrulefill}}

\newcommand{\euV}{\mathscr{V}}
\newcommand{\euE}{\mathscr{E}}
\newcommand{\euP}{\mathscr{P}}
\newcommand{\eue}{\mathbf{e}}
\newcommand{\K}{\mathbb{K}}

\newcommand{\F}{\mathbb{F}}
\newcommand{\dbN}{\mathbb{N}}
\newcommand{\N}{\dbN}

\newcommand{\Z}{\mathbb{Z}}
\newcommand{\calG}{\mathcal{G}}
\newcommand{\rmH}{\mathrm{H}}
\newcommand{\bfH}{\mathbf{H}}
\newcommand{\bfLam}{\mathbf{\Lambda}}

\newcommand{\op}{\mathrm{op}}

\DeclareFontFamily{U}{wncy}{}
\DeclareFontShape{U}{wncy}{m}{n}{<->wncyr10}{}
\DeclareSymbolFont{mcy}{U}{wncy}{m}{n}
\DeclareMathSymbol{\Sha}{\mathord}{mcy}{"58}
\DeclareMathSymbol{\sha}{\mathord}{mcy}{"78}

\begin{document}

	\newtheorem{thm}{Theorem}[section]
	\newtheorem{cor}[thm]{Corollary}
	\newtheorem{lem}[thm]{Lemma}
	\newtheorem{prop}[thm]{Proposition}
	\newtheorem{defin}[thm]{Definition}
	\newtheorem{exam}[thm]{Example}
	\newtheorem{examples}[thm]{Examples}
	\newtheorem{rem}[thm]{Remark}
	\newtheorem{case}{\sl Case}
	\newtheorem{claim}{Claim}
	\newtheorem{fact}[thm]{Fact}
	\newtheorem{question}[thm]{Question}
	\newtheorem{conj}[thm]{Conjecture}
	\newtheorem*{notation}{Notation}
	\swapnumbers
	\newtheorem{rems}[thm]{Remarks}
	\newtheorem*{acknowledgment}{Acknowledgment}
	
	\newtheorem{questions}[thm]{Questions}
	\numberwithin{equation}{section}

	
	\title[Oriented pro-$\ell$ RAAGs and maximal pro-$\ell$ Galois groups]{Oriented right-angled Artin pro-$\ell$ groups \\ and maximal pro-$\ell$ Galois groups}
	\author[S.~Blumer]{Simone Blumer}
	\author[C.~Quadrelli]{Claudio Quadrelli}
	\author[Th.~Weigel]{Thomas S. Weigel}
	\address{Department of Mathematics and Applications, University of Milano Bicocca, Milan (Italy)}
	\email{s.blumer@campus.unimib.it}
	\address{Department of Science and High-Tech, University of Insubria, Como (Italy)}
	\email{claudio.quadrelli@uninsubria.it}
	\address{Department of Mathematics and Applications, University of Milano Bicocca,  Milan (Italy)}
	\email{thomas.weigel@unimib.it}
	\date{\today}
	
	\dedicatory{Dedicated to the memory of Avinoam Mann.}
	\begin{abstract}
		For a prime number $\ell$ we introduce and study oriented right-angled Artin pro-$\ell$ groups $G_{\Gamma,\lambda}$(oriented pro-$\ell$ RAAGs for short) associated to a finite oriented graph $\Gamma$ and a continuous group homomorphism $\lambda\colon\Z_\ell\to\Z_\ell^\times$.
		We show (cf. Thm.~\ref{thm:main}) that an oriented pro-$\ell$ RAAG $G_{\Gamma,\lambda}$
		is a Bloch-Kato pro-$\ell$ group if, and only if, 
		$(G_{\Gamma,\lambda},\theta_{\Gamma,\lambda})$ is an oriented pro-$\ell$ group of elementary type
		generalizing a recent result of I. Snopche and P.~Zalesski\u i (cf. \cite{SZ:RAAGs}). 
		Here $\theta_{\Gamma,\lambda}\colon G_{\Gamma,\lambda}\to\Z_p^\times$ denotes the canonical $\ell$-orientation on 
		$G_{\Gamma,\lambda}$.
		We invest some effort in order to show that oriented right-angled Artin pro-$\ell$ groups share many properties with 
		right-angled Artin pro-$\ell$-groups or even discrete RAAG's, e.g., if $\Gamma$ is a specially oriented chordal graph, then
		$G_{\Gamma,\lambda}$ is coherent (cf. Thm.~\ref{thm:chordal}(ii)) generalizing a result of C.~Droms 
		(cf.~\cite{droms:coh}) . Moreover, in this case $(G_{\Gamma,\lambda},\theta_{\Gamma,\lambda})$ has the Positselski-Bogomolov property 
		(cf. Thm.~\ref{thm:chordal}(i)) generalizing a result of
		H.~Servatius, C.~Droms and  B.~Servatius for discrete RAAG's (cf.~\cite{serv}).
		If $\Gamma$ is a specially oriented chordal graph and $\image(\lambda)\subseteq 1+4\Z_2$ in case that $\ell=2$, then $H^\bullet(G_{\Gamma,\lambda},\F_\ell)
		\simeq \Lambda^\bullet(\ddot{\Gamma}^{\op})$ (cf. Thm.~\ref{thm:chordal}(iii)) generalizing a well known result of M.~Salvetti (cf.~\cite{salv:coh}).
	\end{abstract}

	

	\maketitle

	\section{Introduction}
	\label{sec:intro}
	
	In a recent paper (see \cite{SZ:RAAGs}) I. Snopce and P.A.~Zalesski\u i showed that for pro-$\ell$ completions of right-angled Artin groups, 
	the following four properties are equivalent: the Bloch-Kato property, 1-cyclotomicity, the elementary type property, and the realizability as a maximal pro-$\ell$ Galois group .
	This theorem can be seen as a pro-$\ell$ analogue of a result of C.~Droms (see \cite{droms:sub}) for discrete right-angled Artin groups, and provides evidence to the conjecture that a finitely generated cyclotomically oriented 
	Bloch-Kato pro-$\ell$ group $(G,\theta)$ is necessarily of elementary type (see \cite[\S~7.5]{qw:cyclotomic}), and therefore to the {\sl Elementary Type Conjecture} formulated by I. Efrat in \cite{efrat:etc1}.
	
	In this paper we generalize the concept of right-angled Artin pro-$\ell$ group to non-necessarily trivially oriented examples.  
	
	By an {\sl oriented graph} $\Gamma=(\euV,\euE)$ we will understand a non-empty {\sl set of vertices} $\euV$ together with a {\sl set of edges} 
	
	$$\euE\subseteq\euV\times\euV\smallsetminus\Delta(\euV),\qquad\text{where }\Delta(\euV)=\{\:(v,v)\:\mid\: v\in\euV\:\}.$$
	
	Thus, for an edge $\eue=(v,w)\in\euE$, its inverse $\beue=(w,v)$ might be contained in $\euE$ (in which case $\eue$ will be said to be an {\sl ordinary edge}), or not (in which case $\eue$ will be said to be {\sl special}). Hence $\euE$ is the disjoint union $\euE_\mathrm{o}\sqcup\euE_\mathrm{s}$ of the sets of ordinary edges and special edges.
	
	In a similar fashion, $\euV$ is the disjoint union $\euV_\mathrm{o}\sqcup\euV_\mathrm{s}$ of the sets of {\sl ordinary vertices} and {\sl special vertices}, satisfying the following condition: the origin (i.e., the first coordinate) of every edge is an ordinary vertex, i.e., a vertex can be special only if either it is isolated or it is the end of a special edge (see \S~\ref{ssec:orientedgraphs} below). 
	
	For an oriented graph $\Gamma=(\euV,\euE)$ and a continuous homomorphism $\lambda\colon \Z_{\ell}\to\Z_{\ell}^\times$
	we call the pro-$\ell$ group $G_{\Gamma,\lambda}$ given by the presentation
	\begin{equation}	\label{def:orRAAG}
		G_{\Gamma,\lambda}=\left\langle\: v\in\euV\:\mid\: \forall\:\eue=(v,w)\in\euE,\: 
		wvw^{-1}=\begin{cases} v &\text{if }\eue\in\euE_\mathrm{o}, \\ v^{\lambda(1)} & \text{if }\eue\in\euE_\mathrm{s}
		\end{cases} \:\right\rangle_{\mathrm{pro-}{\ell}}
	\end{equation}
	the {\sl oriented right-angled Artin pro-$\ell$ group} (oriented pro-$\ell$ RAAG for short) associated to $\Gamma$ and $\lambda$.
	By definition, every oriented right-angled Artin pro-$\ell$ group $G_{\Gamma,\lambda}$ carries the orientation 
	$\theta_{\Gamma,\lambda}\colon G_{\Gamma,\lambda}\to \Z_{\ell}^\times$ given by $\theta_{\Gamma,\lambda}(v)=1$ if $v\in\euV_\mathrm{o}$, and $\theta_{\Gamma,\lambda}(v)=\lambda(1)$ if $v\in\euV_\mathrm{s}$.
	Thus, if $\Gamma$ is an oriented graph without special edges, then $G_{\Gamma,\lambda}$ is just a right-angled Artin pro-$\ell$ group, and $\theta_{\Gamma,\lambda}=\bone$ is the constant $1$-function.
	In spite of the rather elementary presentation \eqref{def:orRAAG}, oriented pro-$\ell$ RAAGs yield surprising variety and flexibility --- for example, among oriented pro-$\ell$ RAAGs one finds free pro-$\ell$ groups, $\ell$-adic analytic groups, and even finite $\ell$-groups.
	
	Our main goal is to prove the following  ``oriented analogue'' of \cite[Thm.~1.2, Thm.~1.5]{SZ:RAAGs}.
	
	\begin{thm}\label{thm:main}
		Let $\Gamma$ be an oriented graph, let $\lambda\colon \Z_{\ell}\to\Z_{\ell}^\times$ be a continuous homomorphism {\rm(}satisfying $\image(\lambda)\subseteq1+4\Z_2$ if $\ell=2${\rm)}, and let $(G_{\Gamma,\lambda},\theta_{\Gamma,\lambda})$ be the oriented pro-$\ell$ RAAG associated to $\Gamma$ and $\lambda$. The following are equivalent.
		\begin{itemize}
			\item[(0)] The oriented graph $\Gamma$ is of elementary type 
			(see Definition~\ref{defin:ETgraphs}).
			\item[(i)] The pro-$\ell$ group $G_{\Gamma,\lambda}$ is isomorphic to the maximal pro-$\ell$ Galois group of a field containing a primitive $\ell$-th root of 1 {\rm(}and also $\sqrt{-1}$ if $\ell=2${\rm)}.
			\item[(ii)] The pro-$\ell$ group $G_{\Gamma,\lambda}$ is a Bloch-Kato pro-$\ell$ group.
			\item[(iii)] The oriented pro-$\ell$ RAAG $(G_{\Gamma,\lambda},\theta_{\Gamma,\lambda})$ is a 1-cyclotomic oriented pro-$\ell$ group.
			\item[(iv)] The oriented pro-$\ell$ RAAG $(G_{\Gamma,\lambda},\theta_{\Gamma,\lambda})$ is an oriented pro-$\ell$ group of elementary type.
			\item[(v)] Every finitely generated closed subgroup of $G_{\Gamma,\lambda}$ is again isomorphic to an oriented pro-$\ell$ RAAG.
		\end{itemize} 
	\end{thm}
	
	Note that condition~(0) in Theorem~\ref{thm:main} gives a combinatorial (and rather immediate) criterion to check whether an oriented pro-$\ell$ RAAG may occur as a maximal pro-$\ell$ Galois group or not.	
	On one handside, Theorem~\ref{thm:main} provides a plethora of brand new examples of pro-$\ell$ groups which do not occur as maximal pro-$\ell$ Galois groups.
	It is worth mentioning that the relations involved in the presentation of an oriented pro-$\ell$ RAAG are elementary --- just an elementary commutator times, possibly, the $\lambda(1)$-th power of a generator ---, especially in comparison with other examples of pro-$\ell$ groups which are known not to occur as maximal pro-$\ell$ Galois groups (see, e.g., \cite{BLMS,cq:noGal}, \cite[\S~9]{cem}, \cite[\S~7]{MT:Gal}), whose presentations require higher commutators.
	
	On the other hand, considering the richness of the family of oriented pro-$\ell$ RAAGs, Theorem~\ref{thm:main} corroborates the Elementary Type Conjecture, and the {\sl Smoothness Conjecture}, formulated by Ch.~De~Clercq and M.~Florence (see \cite{dcf:lift}).
	
	In fact, there are several other conjectures concerning the structure of the maximal pro-$\ell$ Galois group $G_{\K}(\ell)$ of a field $\K$ containing a primitive $\ell$-th root of 1, and its $\F_\ell$-cohomology, e.g., L.~Positselski's formulation of {\sl Bogomolov Conjecture}, which predicts that $G_{\K}(\ell)$ is a free-by-locally uniform pro-$\ell$ group (see \cite{pos:k} and \cite{qw:bogomolov}); J.~Mina\v{c} and N.D.~T\^an's {\sl Massey Vanishing Conjecture}, which predicts that $G_{\K}(\ell)$ satisfies the $n$-Massey product vanishing property for every $n>2$ (see \cite{MT:conj}); and the {\sl Universal Koszulity Conjecture}, formulated by J.~Mina\v{c} et al., which predicts that $\bfH^\bullet(G_{\K}(\ell),\F_\ell)$ is a universally Koszul algebra (see \cite{MPPT}).
	These three properties are satisfied by oriented pro-$\ell$ groups of elementary type (see respectively \cite[Thm.~1.2]{qw:bogomolov}, \cite[Thm.~1.3]{cq:massey}, and \cite[\S~1]{MPPT}).
	Therefore, from Theorem~\ref{thm:main} one concludes the following
	providing evidence for the three aforementioned conjectures.
	
	\begin{cor}\label{cor:main}
		Let $\Gamma$ be an oriented graph and let $\lambda\colon \Z_{\ell}\to\Z_{\ell}^\times$ be a continuous homomorphism {\rm(}satisfying $\image(\lambda)\subseteq1+4\Z_2$ if $\ell=2${\rm)}.
		If $G_{\Gamma,\lambda}\simeq G_{\K}(\ell)$ for some field $\K$ containing a primitive $\ell$-th root of 1, then
		\begin{itemize}
			\item[(i)] $G_{\Gamma,\lambda}$ has the Bogomolov-Positselski property, i.e., it is the Frattini pro-$\ell$ cover of a locally uniform pro-$\ell$ group with free pro-$\ell$ kernel;
			\item[(ii)] for every $n>2$, $G_{\Gamma,\lambda}$ satisfies the $n$-Massey product vanishing property, i.e., every non-empty $n$-fold Massey product associated to an $n$-tuple of elements of $\rmH^1(G_{\Gamma,\lambda},\F_\ell)$ contains 0;
			\item[(iii)] the algebra $\bfH^\bullet(G_{\Gamma,\lambda},\F_\ell)$ is universally Koszul.
		\end{itemize}
	\end{cor}
	
	Subsequently, we will focus on oriented pro-$\ell$ RAAGs associated to {\sl chordal} oriented graphs.
	
	Recall that an oriented graph $\Gamma$ is said to be chordal if it does not contain cycles other than triangles as induced subgraphs. 
	
	Moreover, an oriented graph will be said to be {\sl specially oriented} if the terminus of a special edge is always a special vertex (see Definition~\ref{defi:specialgraphs}). Note that an oriented graph of elementary type is always specially oriented and chordal, but not vice-versa.
	
	For oriented pro-$\ell$ RAAGs associated to chordal specially oriented graphs we prove the following.
	
	\begin{thm}\label{thm:chordal}
		Let $\Gamma$ be a chordal specially oriented graph, let $\lambda\colon \Z_{\ell}\to\Z_{\ell}^\times$ be a continuous homomorphism {\rm(}satisfying $\image(\lambda)\subseteq1+4\Z_2$ if $\ell=2${\rm)}, and let $(G_{\Gamma,\lambda},\theta_{\Gamma,\lambda})$ be the oriented pro-$\ell$ RAAG associated to $\Gamma$ and $\lambda$.
		Then:
		\begin{itemize}
			\item[(i)] $G_{\Gamma,\lambda}$ has the Bogomolov-Positselski property;
			\item[(ii)] every finitely generated closed subgroup of $G_{\Gamma,\lambda}$ is of type $FP_\infty$. In particular, $G$ is coherent;
			\item[(iii)] the $\F_\ell$-cohomology algebra $\bfH^\bullet(G_{\Gamma,\lambda},\F_\ell)$ is quadratic --- in particular, it is isomorphic to the exterior Stanley-Reisner $\F_\ell$-algebra $\bfLam_\bullet(\ddot\Gamma^{\op})$ associated to $\ddot\Gamma^{\op}$.
			
		\end{itemize}
	\end{thm}
	
	Note that Theorem~\ref{thm:chordal}--(i) can be seen as an oriented pro-$\ell$ analogue of a theorem of H.~Servatius, C.~Droms and B.~Servatius (see \cite{serv}). One may speculate whether for a
	generalized pro-$\ell$ RAAG's $G_{\Gamma,\lambda}$ based on a specially oriented graph $\Gamma$
	the Bogomolov-Positselski property of $(G_{\Gamma,\lambda},\theta_{\Gamma,\lambda})$ is equivalent to the
	chordality of $\Gamma$.
	Theorem~\ref{thm:chordal}--(ii) can be seen as an oriented pro-$\ell$ analogue of another theorem of C.~Droms (see \cite{droms:coh}). As before one is tempted to believe that the coherence property of $G_{\Gamma,\lambda}$
	for a specially oriented graph $\Gamma$ is equivalent to the
	chordality of $\Gamma$. 
	Theorem~\ref{thm:chordal}--(iii) can be seen as an oriented analogue of a result of M.~Salvetti on discrete RAAG's
	(cf.~\cite{salv:coh}) which holds independently of the chordality of the underlying graph. Therefore, one may conjecture
	that Theorem~\ref{thm:chordal}--(iii) holds for all $G_{\Gamma,\lambda}$ based on specially oriented graphs $\Gamma$ (see \ref{q:cring} below).
	
	One reason why it is interesting to investigate right-angled Artin pro-$\ell$ groups $G_\Gamma$ (with $\Gamma$ a graph without special vertices and edges) in a Galois-theoretic context
	is the fact that  $\bfH^\bullet(G_{\Gamma},\F_\ell)$ is isomorphic to the exterior Stanley-Reisner $\F_\ell$-algebra $\bfLam_{\bullet}(\Gamma^{\op})$ associated to $\Gamma$, and thus it is quadratic and Koszul (see, e.g., \cite[\S~1.2]{papa}).
	Moreover, one has that the {\sl $\N_0$-graded $\F_\ell$-group algebra} $\gr_\bullet(G_\Gamma)$ 
	associated to the augmentation filtration
	is isomorphic to the Cartier-Foita $\F_{\ell}$-algebra $\mathbf{R}_\Gamma$ associated to $\Gamma$, and thus it is quadratic and Koszul as well.
	In particular, it is the {\sl quadratic dual} of $\bfH^\bullet(G_{\Gamma},\F_\ell)$ (see \cite[Thm.~1.2]{bartholdi}).
	
	The same phenomenon occurs for oriented pro-$\ell$ groups $(G,\theta)$ of elementary type: if $G$ is such a pro-$\ell$ group, then both algebras $\bfH^\bullet(G,\F_{\ell})$ and $\gr_\bullet(G)$ are Koszul and quadratic dual to each other (see \cite[Thm.~A--B]{MPQT}).
	Moreover, the same phenomenon is conjectured to hold for finitely generated maximal pro-$\ell$ Galois groups of fields containing a primitive $\ell$-th root of 1, see \cite[Conj.~1.3 and Ques.~1.5]{MPQT}. 
	We suspect that the same holds also for oriented pro-$\ell$ RAAGs associated to specially oriented graphs. Altogether we ask the following.
	
	\begin{question} \label{q:cring}
		Let $\Gamma$ be a specially oriented graph, let $\lambda\colon \Z_{\ell}\to\Z_{\ell}^\times$ be a continuous homomorphism {\rm(}satisfying $\image(\lambda)\subseteq1+4\Z_2$ if $\ell=2${\rm)}, and let $(G_{\Gamma,\lambda},\theta_{\Gamma,\lambda})$ be the oriented pro-$\ell$ RAAG associated to $\Gamma$ and $\lambda$.
		\begin{enumerate}
			\item Is it true that the $\N_0$-graded $\F_\ell$-group algebra $\gr_\bullet(G_{\Gamma,\lambda})$ is quadratic --- and thus isomorphic to the Cartier-Foita $\F_\ell$-algebra associated to $\Gamma$?
			\item Is it true that the cohomology ring $\bfH^\bullet(G_{\Gamma,\lambda},\F_{\ell})$ is isomorphic to the exterior Stanley-Reisner $\F_{\ell}$-algebra $\bfLam_{\bullet}(\ddot\Gamma^{\op})$?
		\end{enumerate}
	\end{question}

	{\small{ \subsection*{Acknowledgments} The authors wish to warmly thank Ido~Efrat, Ilir~Snopce, Matteo~Vannacci and Pavel~A.~Zalesski\u{i} for several helpful discussions.
	}}
	
	
	\section{Graphs}
	\label{sec:graphs}

	Although there is no standard notion of a {\sl graph}, all notions known to the authors
	have one common feature: a graph consists of a pair of sets $(\euV,\euE)$,  {\sl a set of vertices} $\euV$,
	and {\sl a set of edges} $\euE$.
	In this paper we make use of two different notions --- {\sl na\"ive} graphs and {\sl oriented} graphs ---, which are now discussed in more detail.
	
	
	\subsection{Na{\"i}ve graphs}\label{ssec:naivegraphs}
	
	A {\sl na\"ive graph} $\Gamma=(\euV,\euE)$ consists of a non-empty set of {\sl vertices} $\euV$ and a set of {\sl edges} $\euE\subseteq \euP_2(\euV)$, where $\euP_2(\euV)$ denotes the set of subsets of $\euV$ of cardinality $2$.
	The na\"ive graph $\Gamma=(\euV,\euE)$ satisfying $\euE=\euP_2(\euV)$ is said to be a {\sl complete graph}. 
	Subgraphs are defined in the obvious way, and a subgraph $\Lambda=(\euV^\prime,\euE^\prime)$ of the na\"ive graph $\Gamma=(\euV,\euE)$ is said to be induced, if
	\begin{equation}
		\label{eq:prop}
		\euE^\prime=\euE\cap\euP_2(\euV').
	\end{equation}
	Moreover, the subgraph is said to be proper if $\euV'\subsetneq\euV.$
	
	Henceforth, we will always consider finite na\"ive graphs, i.e., na\"ive graphs with a finite number of vertices.

	Recall that a {\sl tree} is a connected na\"ive graph $\mathrm{T}=(\euV,\euE)$, $\euV\neq\varnothing$, without circuits as subgraphs (cf. \cite[\S~2.2]{serre:trees}).
	Every connected nai\"ve graph $\Gamma=(\euV,\euE)$ contains a {\sl maximal subtree}, i.e., a tree $\mathrm{T}=(\euV_\mathrm{T},\euE_\mathrm{T})$ which is a subgraph of $\Gamma$ such that $\euV_\mathrm{T}=\euV$ (cf. \cite[Prop.~2.11]{serre:trees}).

	\begin{exam}\label{ex:C4L3}\rm
		The na\"ive graphs $\mathrm{C}_4$ and $\mathrm{L}_3$ are respectively the na\"ive graphs with geometric realization
		
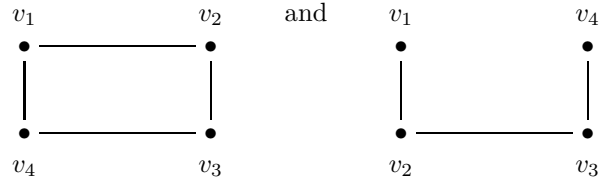
\[
		\xymatrix@R=0.5pt{v_1 && v_2 \\ \bullet\ar@{-}[rr]\ar@{-}[dddd] && \bullet \\ \\ \\
			\\ \bullet && \bullet\ar@{-}[uuuu]\ar@{-}[ll] \\ v_4 && v_3}
		\qquad\text{and}\qquad
		\xymatrix@R=0.5pt{v_1 && v_4 \\ \bullet\ar@{-}[dddd] && \bullet \\ \\ \\
			\\ \bullet && \bullet\ar@{-}[uuuu]\ar@{-}[ll] \\ v_2 && v_3}
		\]
		i.e., $\mathrm{C}_4$ is a square and $\mathrm{L}_3$ is a line of length 3.
	\end{exam}
	
	A finite complete subgraph $\Xi=(\euV^\prime,\euE^\prime)$ of a na\"ive graph $\Gamma=(\euV,\euE)$ is said
	to be a {\sl $|\euV^\prime|$-clique}. Note that such a subgraph is always induced.
	For $n\geq1$, an $n$-clique $\Xi$ of $\Gamma$ is said to be maximal if there are no $(n+1)$-cliques of $\Gamma$ containing $\Xi$ as subgraph.
	
	\begin{defin}\label{defi:derivedgraph}\rm
		Let $\Gamma=(\euV,\euE)$ be a na\"ive graph. 
		The {\sl clique-graph} of $\Gamma$ is the graph
		$\Upsilon(\Gamma)=(\mathrm{\bf mx}(\Gamma),\mathrm{\bf mx}^2(\Gamma))$
		with 
		\begin{equation}
			\label{eq:derGr1}
			\begin{aligned}
				\mathrm{\bf mx}(\Gamma)&=\{\,\Xi\subseteq\Gamma\mid \Xi\ \text{a maximal clique in}\ \Gamma\,\}\\
				\mathrm{\bf mx}^2(\Gamma)&=\{\,\{\Xi,\Xi^\prime\}\mid \Xi,\Xi^\prime\in\mathrm{\bf mx}(\Gamma):\euV(\Xi)\cap\euV(\Xi^\prime)\not=\varnothing\,\}.
			\end{aligned}
		\end{equation}
		A maximal subtree $\mathrm{T}_{\Upsilon(\Gamma)}$ of $\Upsilon(\Gamma)$ is said to have the {\sl clique-intersection property} if, for every pair of distinct maximal cliques $\Xi,\Xi'\in \mathrm{\bf mx}(\Gamma)$, with path
		\[\begin{split}
			P_{\Xi,\Xi'}&=\left(\euV({P_{\Xi,\Xi'}}),\euE({P_{\Xi,\Xi'}})\right)\\
			\euV(P_{\Xi,\Xi'}) &=\{\:\Xi_1=\Xi,\:\Xi_2,\:\ldots,\:\Xi_{r-1},\Xi_r=\Xi'\:\}\subseteq\mathrm{\bf mx}(\Gamma),\\
			\euE(P_{\Xi,\Xi'}) &=\left\{\:\{\Xi_1,\Xi_2\},\:\ldots,\:\{\Xi_{r-1},\Xi_r\}\:\right\}\subseteq \euE(\mathrm{T}_{\Upsilon(\Gamma)}),
		\end{split}\]
		connecting them in the tree $\mathrm{T}_{\Upsilon(\Gamma)}$,
		the clique $$\Xi\cap\Xi'=(\euV(\Xi)\cap\euV(\Xi'),\euE(\Xi)\cap\euE(\Xi'))$$ is a subgraph of $\Xi_i$ for every $i=1,\ldots,r$ (cf. \cite[\S~3.1]{chordalgraphs}).
	\end{defin}
	
	For na\"ive graphs one has the following construction.
	Let $\Gamma_1=(\euV_1,\euE_1)$ and $\Gamma_2=(\euV_2,\euE_2)$ be two na\"ive graphs,
	with a common induced proper subgraph 
	$\Lambda=(\euV',\euE')$.
	The {\sl patching} of $\Gamma_1,\Gamma_2$ along $\Lambda$ is the graph $\Gamma=(\euV,\euE)$ with
	\[
	\euV=\euV_1\cup\euV_2,\qquad\euE=\euE_1\cup\euE_2,
	\]
	where we identify the vertices lying in $\euV_1\cap\euV'$ and in $\euV_2\cap\euV'$, and the edges lying in $\euE_1\cap\euE'$ and in $\euE_2\cap\euE'$.

	\subsection{Oriented graphs}\label{ssec:orientedgraphs}

	An {\sl oriented graph} $\Gamma=(\euV,\euE)$ consists of a non-empty set of {\sl vertices} $\euV$, partitioned as a disjoint union $\euV=\euV_\mathrm{s}\sqcup\euV_\mathrm{o}$; and a set of {\sl edges} 
	$$\euE\subseteq\euV\times\euV\smallsetminus\Delta(\euV),$$
	where $\Delta(\euV)=\{\,(v,v)\mid v\in\euV\,\}$ denotes the diagonal in $\euV\times\euV$.
	By definition, every oriented graph $\Gamma=(\euV,\euE)$ comes equipped with two maps, the {\sl origin} $o\colon\euE\to\euV$ given by the projection on the first coordinate, and the {\sl terminus} $t\colon\euE\to\euV$, given by the
	projection on the second coordinate.
	One has a partition $\euE=\euE_\mathrm{s}\sqcup\euE_\mathrm{o}$, where 
	\[
	\euE_\mathrm{s} = \{\:\eue\in\euE\:\mid\: (t(\eue),o(\eue))\notin\euE\:\}
	\]
	is the set of {\sl special edges}, while $\euE_\mathrm{o}$ is the set of {\sl ordinary edges}.
	Analogously, the elements of $\euV_\mathrm{s}$ and $\euV_\mathrm{o}$ are called special vertices and ordinary vertices respectively, and they must satisfy the following condition: if $\eue\in\euE$, then $o(\eue)\in\euV_\mathrm{o}$. 
	
	Every oriented graph $\Gamma=(\euV,\euE)$ defines a na\"ive graph
	$\ddot{\Gamma}=(\euV,\ddot{\euE})$, where
	\begin{equation}
		\label{eq:ddE}
		\ddot{\euE}=\{\,\{o(\eue),t(\eue)\}\mid\eue\in\euE\,\}.
	\end{equation}
	The notions of subgraphs, induced subgraphs, cliques and patching of oriented graphs are defined in the obvious way.
	Henceforth, we will consider only finite oriented graphs.
	
	\begin{exam}\label{ex:orientedgraph0}\rm
		Let $\Gamma=(\euV,\euE)$ be an oriented graph.
		We realize a special edge $(v,w)\in\euE_\mathrm{s}$ as an arrow originating at $v$ and pointing at $w$.
		If $(v,w),(w,v)\in\euE$, then we realize them as a unique ``unoriented'' edge joining $v$ and $w$.
		For example, the two pictures
		
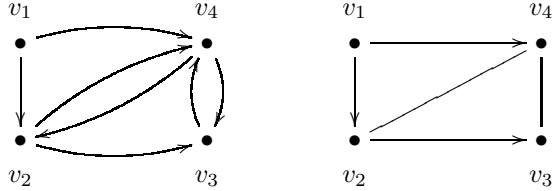
\[
		\xymatrix@R=1.5pt{  v_1  && v_4 \\ \bullet\ar@/^/[rr]\ar[dddd] && \bullet\ar@/^/[lldddd]\ar@/^/[dddd]  \\ 
			\\ \\ \\  \bullet\ar@/_/[rr]\ar@/^/[uuuurr] && \bullet\ar@/^/[uuuu] \\ v_2 && v_3 }\qquad\qquad
		\xymatrix@R=1.5pt{  v_1  && v_4 \\ \bullet\ar[rr]\ar[dddd] && \bullet\ar@{-}[lldddd]\ar@{-}[dddd]  \\ 
			\\ \\ \\  \bullet\ar[rr] && \bullet \\ v_2 && v_3 }
		\]
		realize the same oriented graph $\Gamma=(\euV,\euE)$, with
		\[
		\euE =\underbrace{\{\:(v_1,v_2),\:(v_1,v_4),\:(v_2,v_3)\:\}}_{\euE_\mathrm{s}}\sqcup
		\underbrace{\{\:(v_2,v_4),\:(v_4,v_2),\:(v_3,v_4),\:(v_4,v_2)\:\}}_{\euE_\mathrm{o}}
		\] and $\euV_\mathrm{o}=\euV$.
		(Henceforth, we will use the second type of realization for oriented graphs.)
	\end{exam}
	
	\begin{rem}\rm
		Let $\Gamma_1=(\euV_1,\euE_1)$ and $\Gamma_2=(\euV_2,\euE_2)$ be two oriented graphs, with $\euV_1=\euV_2=\{v,w\}$, $\euE_1=\euE_2=\varnothing$, and such that both $v,w$ are ordinary vertices of $\Gamma_1$, but $v$ is a special vertex of $\Gamma_2$.
		Then $\Gamma_1$ and $\Gamma_2$ are not the same oriented graph, even if their geometric realizations are equal.
	\end{rem}


	

	
	\subsection{Specially oriented graphs}
	
	\begin{defin}\label{defi:specialgraphs}\rm
		An oriented graph $\Gamma=(\euV,\euE)$ with decompositions $\euV=\euV_\mathrm{o}\sqcup\euV_\mathrm{s}$ and $\euE=\euE_\mathrm{o}\sqcup\euE_\mathrm{s}$ is said to be a {\sl specially oriented graph} (or just a {\sl special graph}), if the terminus of every special edge is a special vertex, i.e., \[
		\euV_{\mathrm{s}}\supseteq\left\{\:t(\eue)\:\mid\:\eue\in\euE_{\mathrm{s}}\:\right\}.
		\] 
		In other words, an oriented graph $\Gamma=(\euV,\euE)$ is special if, and only if, 
		$\Gamma$ does not contain a subgraph $\Gamma'$ whose geometric realization is either
		
		\begin{equation}\label{eq:special bad}
			\xymatrix@R=1.5pt{ y & x & z \\ \bullet\ar@{-}[r] & \bullet & \bullet\ar[l]}
			\qquad\text{or}\qquad
			\xymatrix@R=1.5pt{  y  & x & z \\ \bullet & \bullet\ar[l] & \bullet\ar[l]  }
		\end{equation}
--- we underline that we do not require $\Gamma'$ to be induced, so that the vertices $y$ and $z$ may be joined by an edge.
	\end{defin}

	\begin{rem}\label{rem:comb}\rm
		
		By definition, a specially oriented graph without special vertices is a combinatiorial graph in the sense of J-P. Serre (cf. \cite[\S~2.1]{serre:trees}).
		
	\end{rem}

	\begin{exam}\label{exam:complete specially oriented graph}\rm
		
		A complete oriented graph $\Gamma=(\euV,\euE)$ is specially oriented if, and only if, there is at most one special vertex and --- if there is one, say $v$ --- $(v,w)\in\euE_\mathrm{s}$ for every $w\in\euV\smallsetminus\{v\}$. 
		
	\end{exam}

	In other words, an oriented graph $\Gamma=(\euV,\euE)$ is special if, and only if, for every special vertex $v\in\euV_\mathrm{s}$ one has $(v,w)\notin\euE$ for every $w\in\euV$.


	
	\subsection{Oriented graphs of elementary type}

	For oriented graphs we have the following two constructions.
	\begin{itemize}
		\item[(a)] Let $\Gamma_1=(\euV_1,\euE_1)$ and $\Gamma_2=(\euV_2,\euE_2)$ be two oriented graphs.
		The {\sl disjoint union} of $\Gamma_1$ and $\Gamma_2$ is the oriented graph
		$$\Gamma_1\sqcup\Gamma_2:=\left(\euV_1\sqcup\euV_2,\euE_1\sqcup\euE_2\right).$$
		\item[(b)] Let $\Gamma=(\euV,\euE)$ be an oriented graph, with $\euV=\euV_\mathrm{s}\sqcup\euV_\mathrm{o}$.
		The {\sl cone graph} with basis $\Gamma$ is the oriented graph 
		$$\nabla(\Gamma):=(\euV(\nabla(\Gamma)),\euE(\nabla(\Gamma))),$$ where
		$\euV(\nabla(\Gamma))=\euV\sqcup\{v\}$, with $v$ a ``new'' ordinary vertex, called the {\sl tip} of the cone, and  
		$$\euE(\nabla(\Gamma))=\euE\sqcup\{(v,w)\mid w\in\euV_\mathrm{s}\}\sqcup\{(v,w),(w,v)\mid w\in\euV_\mathrm{o}\}.$$
		Observe that $\euV(\nabla(\Gamma))_\mathrm{s}=\euV_\mathrm{s}$.
	\end{itemize}
	
	\begin{rem}\label{rem:special closed operations}\rm
		It is straightforward to see that if $\Gamma_1$ and $\Gamma_2$ are two special graphs, then also the disjoint union $\Gamma_1\sqcup\Gamma_2$ is a special graph.
		Similarly, if $\Gamma$ is a special graph, then also the cone $\nabla(\Gamma)$ is a special graph. 
	\end{rem}
	
	The following definition generalizes the notion of {\sl graphs of elementary type} given in \cite[\S~3.3]{ac:RAAGs}.
	
	\begin{defin}\label{defin:ETgraphs}\rm
		The family of {\sl special graphs of elementary type} is the smallest family of oriented graphs containing graphs consisting of a single vertex, and such that:
		\begin{itemize}
			\item[(a)] if $\Gamma_1,\Gamma_2$ are special graphs of elementary type, then also their disjoint union $\Gamma_1\sqcup\Gamma_2$ is a special graph of elementary type. 
			\item[(b)] if $\Gamma$ is a special graph of elementary type, then also the cone $\nabla(\Gamma)$ is a special graph of elementary type.
		\end{itemize}
	\end{defin}
	
	\begin{exam}\rm
		Any special complete graph is of elementary type, as it may be constructed as iterated cone --- starting from the unique special vertex, if there is one.
		For instance, the displayed graph may be obtained by iterating three times the cone-construction by starting with the single special vertex $v$, i.e., $\Gamma=\nabla(\nabla(\nabla(\Gamma')))$, where $\Gamma'=(\{v\},\varnothing)$ and $v$ is a special vertex.
		
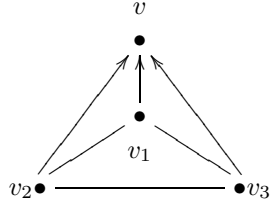
\[ 	\xymatrix@R=1.5pt{&v&\\
			&\bullet&\\
			&&\\
			&&\\
			&\bullet\ar[uuu]&\\
			&v_1&\\
			v_2\bullet\ar@{-}[uur]\ar[uuuuur]\ar@{-}[rr]&&\bullet\ar[uuuuul]\ar@{-}[uul]v_3}\]
		
	\end{exam}
	
	\begin{rem}\label{rem:ETgraph conn}\rm
		\begin{itemize}
			\item[(a)] Since an oriented graph consisting of a single vertex is specially oriented, by Remark~\ref{rem:special closed operations} an oriented graph of elementary type is specially oriented.
			\item[(b)] If a connected oriented graph $\Gamma=(\euV,\euE)$ is of elementary type, then either $\vert\euV\vert=1$ or one has $\Gamma=\nabla(\Gamma')$ for some proper induced subgraph $\Gamma'\subseteq\Gamma$, and therefore there is an ordinary vertex of $\Gamma$ which is joined to all other vertices.
		\end{itemize}
	\end{rem}

	\begin{exam}\label{exam:ETgraphs}\rm
		\begin{itemize}
			\item[(a)] By Remark~\ref{rem:ETgraph conn}, if $\Gamma$ is a special graph such that either $\ddot\Gamma=\mathrm{C}_4$ or $\ddot\Gamma=\mathrm{L}_3$, then $\Gamma$ is not a special graph of elementary type, as $\Gamma$ is connected but there are no vertices which are joined to each other vertex.
			\item[(b)] The special graph $\Lambda_{\mathrm{s}}=(\euV,\euE)$ with geometric realization
			
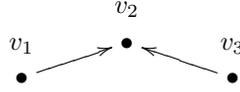
\[ 	\xymatrix@R=1.5pt{& v_2 & \\ v_1 &\bullet& v_3 \\ \bullet\ar[ur] && \bullet\ar[ul] } \]
			is not of elementary type, as it is connected but $v_2$ is special, so that $\Lambda_{\mathrm{s}}$ does not decompose as a cone by Remark~\ref{rem:ETgraph conn}--(b).
		\end{itemize}
	\end{exam}
	
	Observe that the inductive definition of oriented graphs of elementary type implies that every induced subgraph of an oriented graph of elementary type is again an oriented graph of elementary type.
	Oriented graphs of elementary type are characterized as follows.
	
	\begin{prop}\label{prop:graphsET}
		Let $\Gamma=(\euV,\euE)$ be an oriented graph.
		Then $\Gamma$ is of elementary type if, and only if, $\Gamma$ is specially oriented and it does not contain an induced subgraph $\Gamma'$ such that either $\ddot{\Gamma}'=\mathrm{C}_4$, $\ddot{\Gamma}'=\mathrm{L}_3$, or $\Gamma'=\Lambda_{\mathrm{s}}$ (cf. Examples \ref{ex:C4L3}, \ref{exam:ETgraphs}--(b)). 
	\end{prop}
	
	\begin{proof}
		If $\Gamma$ is not specially oriented, then it is not of elementary type by Remark~\ref{rem:ETgraph conn}--(a).
		If $\Gamma$ contains an induced subgraph $\Gamma'$ such that either $\ddot \Gamma'= \mathrm{C}_4$ or $\ddot \Gamma'=\mathrm{L}_3$, or $\Gamma'=\Lambda_\mathrm{s}$, then $\Gamma'$ is not of elementary type by Example~\ref{exam:ETgraphs}--(b)--(c), and thus $\Gamma$ is not of elementary type as well.
		
		Conversely, assume that $\Gamma$ contains no such induced subgraphs, and that it is specially oriented.
		By Remark~\ref{rem:ETgraph conn}--(b), without loss of generality we may assume that $\Gamma$ is connected.
		We proceed by induction on $|\euV|$.
		If $|\euV|\leq 2$, then clearly $\Gamma$ is of elementary type.
		Now, the na\"ive graph $\ddot\Gamma=(\euV,\ddot\euE)$ contains no induced subgraphs $\ddot\Gamma'$ equal to $\mathrm{C}_4$ or $\mathrm{L}_3$.
		By \cite{wolk}, $\ddot\Gamma$ has a \emph{central vertex}, i.e., a vertex $w\in\euV$ such that $\{w,v\}\in\ddot\euE$ for every $v\in\euV\smallsetminus\{w\}$.
		We have two cases.
		
		\medskip
		\noindent {\sl Case (1).}  If $w\in\euV_\mathrm{o}$, then $\Gamma=\nabla(\Gamma')$, with $\Gamma'$ the induced subgraph of $\Gamma$ with vertices $\euV\smallsetminus\{w\}$, which is of elementary type by induction --- and thus also $\Gamma$ is of elementary type.
		\medskip
		
		\noindent {\sl Case (2).} If $w\in\euV_\mathrm{s}$, then for every $v\in\euV\smallsetminus\{w\}$ one has $(v,w)\in\euE$, $(w,v)\notin\euE$, and $\euV_\mathrm{o}=\euV\smallsetminus\{w\}$, as $\Gamma$ is specially oriented (cf. Remark~\ref{rem:ETgraph conn}--(a)).
		Let $v_1,v_2$ be two distinct ordinary vertices of $\Gamma$.
		If $(v_1,v_2),(v_2,v_1)\notin\euE$, then the induced subgraph of $\Gamma$ with vertices $v_1,v_2,w$ is isomorphic to $\Lambda_{\mathrm{s}}$, a contradiction.
		On the other hand, $(v_1,v_2),(v_2,v_1)\notin\euE_{\mathrm{s}}$, as $v_1,v_2\in\euV_{\mathrm{o}}$.
		Therefore, $(v_1,v_2)\in\euE_{\mathrm{o}}$, and hence $\Gamma$ is a complete specially oriented graph --- in particular, $\Gamma$ is of elementary type by Example~\ref{exam:ETgraphs}--(a).
	\end{proof}
	
	
	\subsection{Graphs of pro-$\ell$ groups}
	
	Let $\Gamma=(\euV,\euE)$ be a connected oriented graph without special edges (i.e., a combinatorial graph in the sense of Serre --- see \cite[\S 2.1]{serre:trees}), and let 
	$$\calG(\Gamma)=\{G(v),G(\eue)\mid v\in\euV,\eue\in\euE\}$$
	be a collection of finitely generated pro-$\ell$ groups, endowed with monomorphisms of pro-$\ell$ groups 
	$$\partial_{\eue}\colon G({\eue})\longrightarrow G(t(\eue))$$
	for every edge $\eue\in\euE$.
	Then $\calG(\Gamma)$ is called a {\sl finite graph of (finitely generated) pro-$\ell$ groups} based on $\Gamma$.
	A finite graph of pro-$\ell$ groups $\calG(\Gamma)$ is said to be reduced if for every edge $\eue\in\euE$, 
	$\partial_{\eue}\colon G({\eue})\to G(t(\eue))$ is an isomorphism.
	
	Suppose that $\Gamma=\mathrm{T}$ is a tree.
	The {\sl fundamental group} $\Pi_1(\calG(\mathrm{T}))$ of a finite graph of pro-$\ell$ groups $\calG(\mathrm{T})$ based on $\mathrm{T}$
	is defined by the following pro-$\ell$ presentation:
	\[
	\Pi_1(\calG(\mathrm{T}))=\left\langle G(v)\mid v\in\euV, \partial_{\eue}(g)=\partial_{\bar\eue}(g)\;\forall\:g\in G(\eue),\eue\in\euE \right\rangle
	\]
	(cf. \cite{zalmel}).
	We have the following \cite[Cor. 4.5]{zalmel}.

	\begin{thm}\label{thm:zalmel}
		Let $U$ be an open subgroup of the fundamental pro-$\ell$ group $G=\Pi_1(\calG,\Gamma)$ of a finite graph of pro-$\ell$ groups $\calG$ based on the graph $\Gamma$, and let $\mathrm T$ be the Bass-Serre tree associated to $\calG$.
		Then $U$ is the fundamental pro-$\ell$ group $U=\Pi_1(\mathcal{U},\Delta)$, with $\Delta=U\backslash\!\backslash\mathrm{T}$, and vertex and edge groups $\mathcal{U}(m)$, $m\in\Delta$, are stabilizers $U_{s(m)}$, where $s\colon \Delta\to\mathrm{T}$ is a connected transversal of $\Delta$ in $\mathrm{T}$.
	\end{thm}
	

	\section{Pro-$\ell$ groups and orientations}\label{sec:cohom}
	
	\subsection{Preliminaries on pro-$\ell$ groups}
	
	We work in the world of pro-$\ell$ groups.
	Henceforth, every subgroup of a pro-$\ell$ group will be tacitly assumed to be closed, and the generators of a subgroup will be intended in the topological sense.
	
	In particular, for a pro-$\ell$ group $G$ and a positive integer $n$, $G^n$ will denote the closed subgroup of $G$ generated by the $n$-th powers of all elements of $G$.
	Moreover, for two elements $g,h\in G$, we set $${}^{g}h=ghg^{-1},\qquad\text{and}\qquad[g,h]={}^gh\cdot h^{-1}.$$
	Given two subgroups $H_1,H_2$ of $G$, $[H_1,H_2]$ will denote the closed subgroup of $G$ generated by all commutators $[h,g]$
	with $h\in H_1$ and $g\in H_2$.
	In particular:
	\begin{itemize}
		\item[(a)] $G'$ will denote the closure of the commutator subgroup $[G,G]$ of $G$;
		\item[(b)] $\Phi(G)=G^{\ell}\cdot G'$ will denote the Frattini subgroup of $G$;
		\item[(c)] $G_{(3)}$ will denote the subgroup of $G$ defined by
		\[ G_{(3)}= \begin{cases} G^{\ell}\cdot[G',G] & \text{if }{\ell}\neq2 \\ G^4\cdot (G')^2\cdot[G',G] &\text{if }{\ell}=2.
		\end{cases} \]
	\end{itemize}
	
	For the properties of Galois cohomology of pro-$\ell$ groups we refer to \cite[Ch.~I]{serre:gal} and to \cite[Ch.~I and Ch.~III--\S~9]{nsw:cohn}.
	
	
	\subsection{Oriented pro-$\ell$ groups}
	Let $G$ be a pro-$\ell$ group. 
	An {\sl orientation} of $G$ is a continuous homomorphism $\theta\colon G\to\Z_{\ell}^\times$, where $\Z_{\ell}^\times$ denotes the group of units of $\Z_{\ell}$.
	Observe that $\image(\theta)\subseteq1+{\ell}\Z_{\ell}$, as the pro-$\ell$ Sylow subgroup of $\Z_{\ell}^\times$ is
	$$1+{\ell}\Z_{\ell}=\{1+{\ell}\lambda,\:\lambda\in\Z_{\ell}\}.$$
	We say that the orientation $\theta$ is {\sl torsion-free} if $\ell$ is odd, or if $\ell=2$ and $\image(\theta)\subseteq1+4\Z_2$.
	
	We call a couple $(G,\theta)$ consisting of a pro-$\ell$ group together with an orientation $\theta\colon G\to\Z_{\ell}^\times$ an {\sl oriented pro-$\ell$ group} (in \cite{efrat:small,eq:kummer} an oriented pro-$\ell$ group is called a {\sl cyclotomic pro-$\ell$ pair}).
	A morphism of oriented pro-$\ell$ groups $(G,\theta)\to(H,\tau)$ is a morphism of pro-$\ell$ groups $\varphi\colon G\to H$ satisfying $\theta=\tau\circ\varphi$.
	
	An oriented pro-$\ell$ group $(G,\theta)$ comes endowed with a continuous left $G$-module $\Z_{\ell}(1)$, which is isomorphic to $\Z_{\ell}$ as an abelian pro-$\ell$ group, and with $G$-action
	\[
	g.\lambda=\theta(g)\cdot \lambda\qquad\text{ for every }g\in G,\:\lambda\in\Z_{\ell}(1).
	\]
	Observe that the $G$-module $\Z_{\ell}(1)/{\ell}$ is simply the trivial module $\Z/{\ell}$, as $\theta(g)\equiv 1\bmod {\ell}$ for all $g\in G$.
	
	In the category of oriented pro-$\ell$ groups one has the following constructions (cf. e.g., \cite[\S~3]{efrat:small}).
	\begin{itemize}
		\item[(a)] If $(G,\theta)$ is an oriented pro-$\ell$ group and $N$ is a normal subgroup of $G$ such that $N\subseteq\kernel(\theta)$, then the quotient $G/N$ yields an oriented pro-$\ell$ group $(G/N,\theta_{/N})$, such that the canonical projection $G\to G/N$ induces a morphism of oriented pro-$\ell$ groups $(G,\theta)\to(G/N,\theta_{/N})$.
		\medskip
		
		\item[(b)] If $(G_1,\theta_1)$ and $(G_2,\theta_2)$ are oriented pro-$\ell$ groups, then the free pro-$\ell$ product of $G_1$ and $G_2$ yields an oriented pro-$\ell$ group
		\[
		(G_1,\theta_1)\amalg_{\hat\ell}(G_2,\theta_2):=(G_1\amalg_{\hat \ell}G_2,\theta),
		\]
		where $\theta$ is the orientation obtained by the universal property of the free product.
		
		\medskip
		\item[(c)] If $(G,\theta)$ is an oriented pro-$\ell$ group and $A$ is an abelian pro-$\ell$ group, then one has an oriented pro-$\ell$ group
		\[
		A\rtimes (G,\theta):=(A\rtimes G,\tilde\theta), 
		\]
		where $gag^{-1}=a^{\theta(g)}$ for every $a\in A$ and $g\in G$, and such that the canonical projection $A\rtimes G\to G$ induces a morphism of oriented pro-$\ell$ groups $(A\rtimes G,\tilde\theta)\to(G,\theta)$.
	\end{itemize}
	
	The most relevant examples of oriented pro-$\ell$ groups arise from maximal pro-$\ell$ Galois groups (cf. \cite[\S~4]{eq:kummer}).
	
	\begin{exam}\label{ex:Galois}\rm
		For a field $\K$, let $\bar\K$ denote a separable closure of $\K$, and let $\mu_{\ell^\infty}$ denote the 
		group of roots of 1 of $\ell$-power order lying in $\bar\K$.
		If $\K$ contains a root of 1 of order $\ell$, then $\mu_{\ell^\infty}$ is contained in the maximal pro-$\ell$-extension
		$\K(\ell)$ of $\K$, and the action of the {\sl maximal pro-$\ell$ Galois group} $G_{\K}(\ell)=\mathrm{Gal}(\K(\ell)/\K)$ of $\K$
		on $\mu_{\ell^\infty}$ fixes the roots of order $\ell$, and induces a natural orientation, the so-called {\sl pro-$\ell$-cyclotomic character} 
		\[
		\hat\theta_{\K,\ell}\colon G_{\K}(\ell)\longrightarrow \Z_\ell^\times,
		\]
		as $\mu_{\ell^\infty}\simeq\Z[\frac{1}{\ell}]/\Z$ and $\Aut(\Z[\frac{1}{\ell}]/\Z)$ is isomorphic to $\Z_\ell^\times$.
		In particular,
		$$\sigma(\zeta)=\zeta^{\hat\theta_{\K,\ell}(\sigma)}\qquad \text{for all }\sigma\in G_{\K}(\ell),\zeta\in\mu_{\ell^\infty}.$$
		Furthermore, one has $\image(\hat\theta_{\K,\ell})=1+\ell^f\Z_\ell$, where $f$ is the positive integer satisfying
		$|\mu_{\ell^\infty}\cap\K^\times|=\ell^f$ in case $\mu_{\ell^\infty}\cap\K^\times$ is finite,  and 
		$\image(\hat\theta_{\K,\ell})=\{1\}$ if $\mu_{\ell^\infty}\subseteq \K^\times$.
		The continuous $G_{\K}(\ell)$-module $\Z_\ell(1)$ induced by the cyclotomic character is called the 1st Tate twist of $\Z_\ell$ (cf. \cite[Def.~7.3.6]{nsw:cohn}), and for every $n\geq1$, $\Z_\ell(1)/\ell^n$ is isomorphic to the $G_{\K}(\ell)$-module of the $\ell^n$-th roots of 1.
	\end{exam}
	
	
	\subsection{$\theta$-abelian oriented pro-$\ell$ groups and locally uniform pro-$\ell$ groups} An oriented pro-$\ell$ group $(G,\theta)$ with torsion-free orientation is said to be {$\theta$-abelian} if $\kernel(\theta)$ is a free abelian pro-$\ell$ group and one has an isomorphism of oriented pro-$\ell$ groups
	\[
	(G,\theta)\simeq \kernel(\theta)\rtimes(G/\kernel(\theta),\theta_{/\kernel(\theta)}).
	\]
	One has the following group-theoretic ``translation'' of $\theta$-abelianity. 
	Recall that a pro-$\ell$ group $G$ is called {\sl locally uniform} if every finitely generated subgroup $H$ of $G$ is uniform --- i.e., $H$ is torsion free and the commutator subgroup $H'$ is contained in $H^\ell$, and also in $H^4$ if ${\ell}=2$ (cf. \cite[\S~3]{cq:bk}).
	One has the following (cf. \cite[Thm.~A]{cq:bk}).
	
	\begin{prop}\label{prop:locunif}
		A pro-$\ell$ group $G$ is locally uniform if, and only if, there exists a torsion-free orientation $\eth_G\colon G\to\Z_{\ell}^\times$ such that $(G,\eth_G)$ is a $\eth_G$-abelian oriented pro-$\ell$ group.
	\end{prop}
	
	In other words, a pro-$\ell$ group $G$ is locally uniform if, and only if, $G$ has a presentation
	\[
	G=\left\langle\:x_0,x_i:i\in I\:\mid\:{}^{x_0}x_i=x_i^{1+\ell^f},\:[x_i,x_j]=1\;\forall \:i,j\in I\:\:\right\rangle
	\]
	for some set $I$ and some $f\in\dbN\cup\{\infty\}$ ($f\geq2$ if ${\ell}=2$), and in this case $\eth_G(x_0)=1+p^f$ and 
	$\eth_G(x_i)=1$ for all $i\in I$.
	We call $\eth_G$ the canonical orientation of $G$.
	Observe that for every subgroup $H\subseteq G$ one has $\eth_H=\eth_G\vert_H$.
	
	\begin{rem}\label{rem:locunif cohom}\rm
		By Lazard's work \cite{lazard}, the $\F_{\ell}$-cohomology algebra $\bfH^\bullet(G,\F_{\ell})$ of a (finitely generated) uniform pro-$\ell$ group $G$ is isomorphic to the exterior algebra generated by $\rmH^1(G,\F_{\ell})$ (cf. e.g., \cite[Thm.~5.1.5]{sw:cohomology}).
		Hence, a locally uniform pro-$\ell$ group is Bloch-Kato.
	\end{rem}

	
	\subsection{Kummerian oriented pro-$\ell$ groups and 1-cyclotomicity}\label{ssec:1cyc}
	
	An oriented pro-$\ell$ group $(G,\theta)$ comes endowed with a distinguished normal subgroup of $G$,
	\[
	K_\theta(G)=\left\langle\:{}^gh\cdot h^{-\theta(g)}\:\mid\:g\in G,\: h\in\kernel(\theta)\:\right\rangle.
	\]
	(cf. \cite{eq:kummer}).
	One has $K_\theta(G)\subseteq \Phi(G)$, and moreover $K_\theta(G)\supseteq\kernel(\theta)'$, so that the quotient $\kernel(\theta)/K_\theta(G)$ is an abelian pro-$\ell$ group.
	Observe that if $\theta$ is trivial (i.e., constantly equal to 1), then $K_\theta(G)=G'$.
	
	One has the following (cf., e.g., \cite[Prop.~2.6 and \S~3.2]{qw:bogomolov}).
	
	\begin{prop}\label{prop:kummer}
		Let $(G,\theta)$ be an oriented pro-$\ell$ group with torsion-free orientation.
		The following are equivalent:
		\begin{itemize}
			\item[(i)] The map
			\begin{equation}\label{eq:kummer surj}
				\rmH^1(G,\Z_{\ell}(1)/{\ell}^n)\longrightarrow \rmH^1(G,\F_{\ell}),
			\end{equation}
			induced by the epimorphism of continuous $G$-modules $\Z_{\ell}(1)/{\ell}^n\to\F_{\ell}$, is surjective for every $n\geq1$. 
			\item[(ii)] The quotient $\kernel(\theta)/K_\theta(G)$ is a free abelian pro-$\ell$ group.
			\item[(iii)] There exists an epimorphism of pro-$\ell$ groups $\varphi \colon G\to\bar G$ where $\bar G$ is locally uniform and $\kernel(\varphi)\subseteq\Phi(G)$, inducing an epimorphism of oriented pro-$\ell$ groups $(G,\theta)\to (\bar G,\eth_{\bar G})$.
		\end{itemize}
		If these conditions hold, then $\kernel(\varphi)=K_\theta(G)$ and
		$$(\bar G,\eth_{\bar G})\simeq \kernel(\theta)/K_\theta(G)\rtimes(G/\kernel(\theta),\theta_{/\kernel(\theta)}).$$
	\end{prop}
	
	An oriented pro-$\ell$ group $(G,\theta)$ with torsion-free orientation $\theta$ satisfying the above equivalent conditions is said to be {\sl Kummerian} (cf. \cite[Def.~3.4]{eq:kummer}).
	If moreover the oriented pro-$\ell$ group $(H,\theta\vert_H)$ is Kummerian for every subgroup $H$ of $G$, then the oriented pro-$\ell$ group $(G,\theta)$ is said to be {\sl 1-cyclotomic} (cf. \cite{qw:cyclotomic}).
	(Observe that in \cite{dcf:lift,cq:galfeat,cq:1smoothBK}, a 1-cyclotomic oriented pro-$\ell$ group is said to be {\sl 1-smooth}.)

	\begin{exam}\label{exam:kummer groups}\rm
		\begin{itemize}
			\item[(a)] If $(G,\theta)$ is a $\theta$-abelian oriented pro-$\ell$ group with torsion-free orientation, then, by Proposition~\ref{prop:kummer}--(iii), it is 1-cyclotomic, as $(H,\theta\vert_H)$ is $\theta\vert_H$-abelian for every subgroup $H$.
			\item[(b)] If $G$ is a free pro-$\ell$ group, then the oriented pro-$\ell$ group $(G,\theta)$ is 1-cyclotomic for any orientation $\theta\colon G\to\Z_\ell^\times$ (cf. \cite[\S~2.2]{qw:cyclotomic}).
			\item[(c)] If $(G,\theta)$ is an oriented pro-$\ell$ group with trivial orientation $\theta\equiv\mathbf{1}$,
			then $(G,\theta)$ is Kummerian if, and only if, the abelianization $G/G'$ is a free abelian pro-$\ell$ group
			(cf. \cite[Example~3.5--(1)]{eq:kummer}).
			Consequently, $(G,\theta)$ is 1-cyclotomic if, and only if, $H/H'$ is a free abelian pro-$\ell$ group for every subgroup $H$ of $G$ (pro-$\ell$ groups satisfying this property are also called {\sl absolutely torsion-free pro-$\ell$ groups}, cf. \cite{wurfel}).
		\end{itemize}
	\end{exam}
	
	1-cyclotomic oriented pro-$\ell$ groups arise naturally in Galois theory (cf. \cite[Thm.~4.2]{eq:kummer} and \cite[\S~2.3]{qw:cyclotomic}).
	
	\begin{thm}\label{thm:galois}
		Let $\K$ be a field containing a primitive $\ell$-th root of 1 (and also $\sqrt{-1}$ if $\ell=2$).
		Then $(G_{\K}(\ell),\hat\theta_{\K,\ell})$ is a torsion-free 1-cyclotomic oriented pro-$\ell$ group.
	\end{thm}
	
	The {\sl pro-$\ell$ version} of De~Clercq-Florence's Smoothness Conjecture (cf. \cite[Conj.~14.25]{dcf:lift}, see also \cite[Conj.~2.10]{cq:galfeat}) states the following.
	
	\begin{conj}\label{conj:smoothness}
		If a pro-$\ell$ group $G$ may be endowed with a 1-cyclotomic orientation $\theta\colon G\to \Z_{\ell}^\times$, then $G$ is {\sl weakly} Bloch-Kato. 
	\end{conj}
	
	The definition of weakly Bloch-Kato pro-$\ell$ group may be found in \cite[Def.~14.23]{dcf:lift} --- for our purposes it is enough to remark that a Bloch-Kato pro-$\ell$ group is (unsurprisingly) also weakly Bloch-Kato (see also \cite[\S~2.3]{cq:galfeat}).
	
	\begin{rem}\label{rem:kummer quotient}\rm
		Let $(G,\theta)$ be a Kummerian oriented pro-$\ell$ group, with $\theta$ a torsion-free orientation. Then by Proposition~\ref{prop:kummer} one has
		\[
		(G/K_\theta(G),\theta_{/K_\theta(G)})\simeq \kernel(\theta)/K_\theta(G)\rtimes(G/\kernel(\theta),\theta_{/\kernel(\theta)})
		\]
		with $\kernel(\theta)/K_\theta(G)$ a free abelian pro-$\ell$ group.
		Now let $N$ be a normal subgroup of $G$ contained in $\kernel(\theta)$ such that the map
		\[
		\mathrm{res}_{G,N}^1\colon \rmH^1(G,\F_\ell)\longrightarrow\rmH^1(N,\F_\ell)^G
		\]
		is injective --- namely, by duality the map $N/N^\ell[N,G]\to G/\Phi(G)$ induced by the inclusion $N\hookrightarrow G$ is injective.
		Then the group $\bar N=NK_\theta(G)/K_\theta(G)$ is an isolated subgroup of the free pro-$\ell$ group $\kernel(\theta)/K_\theta(G)$.
		Therefore, 
		\[
		\frac{\kernel(\theta)}{K_\theta(G)}\simeq \bar N\times A\qquad\text{and}
		\qquad \frac{\kernel(\theta_{/N})}{K_{\theta_{/N}}(G/N)}\simeq \frac{\kernel(\theta)/K_\theta(G)}{\bar N}\simeq A,
		\]
		where $A$ is a free pro-$\ell$ group.
		Thus, also the oriented pro-$\ell$ group $(G/N,\theta_{/N})$ is Kummerian by Proposition~\ref{prop:kummer}.
	\end{rem}
	
	
	\subsection{Oriented pro-$\ell$ groups of elementary type}\label{ssec:ETC}
	
	The following definition is due to I.~Efrat (cf. \cite[\S~3]{efrat:small}, see also \cite[\S~7.5]{qw:cyclotomic}).
	
	\begin{defin}\label{defin:ET}\rm
		The family of oriented pro-$\ell$ groups of {\sl elementary type} is the smallest class of oriented pro-$\ell$ groups such that
		\begin{itemize}
			\item[(a)] the oriented pro-$\ell$ group $(\{1\},\mathbf{1})$ consisting of the trivial group and the trivial orientation is of elementary type, as well as the oriented pro-$\ell$ group $(\Z_{\ell},\lambda)$ for any linear orientation $\lambda\colon \Z_{\ell}\to\Z_{\ell}^\times$;
			\item[(b)] any oriented pro-$\ell$ group $(G,\theta)$ with $G$ a {\sl Demushkin group} and $\theta\colon G\to\Z_{\ell}^\times$ its canonical orientation (cf. \cite[Thm.~4]{labute:demushkin} and \cite[\S~5.3]{qw:cyclotomic});
			\item[(c)] if $(G_1,\theta_1)$ and $(G_2,\theta_2)$ are two oriented pro-$\ell$ groups of elementary type, then also the free product $$(G_1,\theta_1)\amalg^{\hat \ell}(G_2,\theta_2)$$
			is an oriented pro-$\ell$ group of elementary type;
			\item[(d)] if $(G,\theta)$ is an oriented pro-$\ell$ group of elementary type and $A$ is a free abelian pro-$\ell$ group, then also the oriented pro-$\ell$ group $A\rtimes(G,\theta)$ is of elementary type.
		\end{itemize}
	\end{defin}
	
	By \cite[Thm.~1.3]{qw:cyclotomic}, if $(G,\theta)$ is an oriented pro-$\ell$ group of elementary type, then $G$ is Bloch-Kato and $(G,\theta)$ is 1-cyclotomic.
	Moreover, one has the following fact (cf. \cite[\S~3.3--3.4]{qw:cyclotomic}).
	
	\begin{fact}\label{prop:subgps etc}
		Let $(G,\theta)$ be an oriented pro-$\ell$ group and let $H\subseteq G$ be a subgroup. 
		\begin{itemize}
			\item[(i)] If $(G,\theta)=(G_1,\theta_1)\amalg^{\hat \ell}(G_2,\theta_2)$, then
			\[ (H,\theta\vert_H)\simeq {\coprod_{i\in I}}^{\hat\ell}(H_i,\tau_i),    \]
			where $H_i$ is a subgroup of either $G_1$ or $G_2$ --- and $\tau_i$ denotes the restriction of $\theta_1$ or $\theta_2$ respectively ---, or $H_i$ is a free pro-$\ell$ group.
			\item[(ii)] If $(G,\theta)=A\rtimes(G_\circ,\theta_\circ)$ for some oriented pro-$\ell$ group $(G_\circ,\theta_\circ)$ and $A\simeq \Z_\ell$, then 
			\[
			(H,\theta\vert_H)\simeq (A\cap H)\rtimes(H_\circ,\theta_\circ\vert_{H_\circ})
			\]
			for some subgroup $H_\circ\subseteq G_\circ$.
		\end{itemize}
		In particular, if $(G,\theta)$ is of elementary type, then for every finitely generated subgroup $H\subseteq G$ the oriented pro-$\ell$ group $(H,\theta\vert_H)$ is again of elementary type.
	\end{fact}
	
	
	Efrat's Elementary Type conjecture predicts that every finitely generated pro-$\ell$ group which is isomorphic to a maximal pro-$\ell$ Galois group of a field containing a primitive $\ell$-th root of unity, endowed with the cyclotomic character, should be completed into an oriented pro-$\ell$ group of elementary type (cf. \cite{efrat:etc1}, see also \cite[Ques.~4.8]{efrat:etc2}, \cite[Conj. 1.2]{efrat:etc3} and \cite[\S~10]{marshall:etc}).

	\begin{conj}\label{conj:ETC}
		Let $\K$ be a field containing a primitive $\ell$-th root of 1, and suppose that $[\K^\times:(\K^\times)^{\ell}]<\infty$ {\rm (}i.e., the maximal pro-$\ell$ Galois group $G_{\K}({\ell})$ is finitely generated{\rm )}.
		Then $(G_{\K}({\ell}),\hat\theta_{\K,\ell})$ is an oriented pro-$\ell$ group of elementary type.
	\end{conj}
	
	
	
	\section{Oriented pro-$\ell$ RAAGs}\label{sec:orRAAGs}

	
	\subsection{Oriented pro-$\ell$ RAAGs}
	From now on a non-trivial torsion-free orientation $\lambda\colon\Z_{\ell}\to\Z_{\ell}^\times$ will be called a {\sl linear} orientation.
	Recall from the Introduction the definition of an oriented pro-$\ell$ RAAG $(G_{\Gamma,\lambda},\theta_{\Gamma,\lambda})$ associated to an oriented graph $\Gamma=(\euV,\euE)$ and a linear orientation $\lambda$.
	The following fact is a direct consequence of the definition of an oriented pro-$\ell$ RAAG.
	
	\begin{fact}\label{rem:RAAG freeprod}
		Let $\Gamma=(\euV,\euE)$ be an oriented graph, and let $\lambda\colon\Z_{\ell}\to\Z_{\ell}^\times$ be a linear orientation.
		\begin{itemize}
			\item[(i)] If $\Gamma=\Gamma_1\sqcup\Gamma_2$, then the associated oriented pro-$\ell$ group $(\calG_{\Gamma,\lambda},\theta_{\Gamma,\lambda})$ decomposes as free pro-$\ell$ product of oriented pro-$\ell$ groups 
			$$(G_{\Gamma_1,\lambda},\theta_{\Gamma_1,\lambda})\amalg^{\hat {\ell}}(G_{\Gamma_2,\lambda},\theta_{\Gamma_2,\lambda}).$$ 
			\item[(ii)] More generally, if $\Gamma$ is the patching of two induced subgraphs $\Gamma_1,\Gamma_2$ along a common induced subgraph $\Lambda$, 
			then then the associated oriented pro-$\ell$ group $(\calG_{\Gamma,\lambda},\theta_{\Gamma,\lambda})$ decomposes as {\sl amalgamated free pro-$\ell$ product} of oriented pro-$\ell$ groups 
			$$(G_{\Gamma_1,\lambda},\theta_{\Gamma_1,\lambda})\amalg_{G_{\Lambda,\lambda}}^{\hat {\ell}}(G_{\Gamma_2,\lambda},\theta_{\Gamma_2,\lambda}),$$
			with amalgam $G_{\Lambda,\lambda}$ (for the definition of amalgamated free pro-$\ell$ product of oriented pro-$\ell$ groups see \cite[Def.~5.4]{qw:bogomolov}).
			\item[(iii)] If $\Gamma=\nabla(\Gamma')$ for some induced subgraph $\Gamma'$, with tip $v\in\euV$, then one has 
			\[
			(G_{\Gamma,\lambda},\theta_{\Gamma,\lambda})\simeq \langle\:v\:\rangle\rtimes (G_{\Gamma',\lambda},\theta_{\Gamma',\lambda}).
			\]
		\end{itemize}
	\end{fact}
	
	
	\subsection{Oriented pro-$\ell$ RAAGs and generalized pro-$\ell$ RAAGs}
	
	Following \cite[\S~5.1]{qsv:quadratic}, we say that an {\sl $\ell$-labelled graph} is an oriented graph $\Gamma=(\euV,\euE)$ such that $\euE=\euE_\mathrm{s}$, endowed with a map $f=(f_o,f_t)\colon \euE_\mathrm{s}\to\ell\Z_\ell\times\ell\Z_\ell$, which we call the {\sl $\ell$-labelling} of $\Gamma$.
	
	The {\sl generalized pro-$\ell$ RAAG} associated to a $\ell$-labelled graph $\Gamma=(\euV,\euE)$ with labelling $f$ is the pro-$\ell$ group with presentation
	\[
	\left\langle\:v\in\euV\:\mid\:[v,w]v^{f_o(\eue)}w^{f_t(\eue)}=1\text{ for }\eue=(v,w)\in\euE\:\right\rangle.
	\]
	Now let $\lambda$ be a linear orientation, and let $\Gamma=(\euV,\euE)$ be any oriented graph, with $\euE=\euE_\mathrm{s}\sqcup\euE_\mathrm{o}$.
	Then the pro-$\ell$ group $G_{\Gamma,\lambda}$ is the generalized pro-$\ell$ RAAG associated to a $\ell$-labelled graph $\Gamma'=(\euV,\euE')$ such that $\ddot\euE'=\ddot\euE$, with $\ell$-labelling
	\[ f(\eue)=\begin{cases}
		(\lambda(1)-1,0) & \text{if }\eue\in\euE_\mathrm{s}\\ (0,0) & \text{if }\eue\in\euE_\mathrm{o}.         \end{cases}\]

	
	\subsection{Examples}
	
	Let $\lambda\colon\Z_{\ell}\to\Z_{\ell}^\times$ be a linear orientation.
	One has the following examples of oriented pro-$\ell$ RAAGs.
	
	\begin{exam}\rm
		If $\Gamma=(\euV,\euE)$ is an oriented graph with $\euE=\euE_\mathrm{o}$ (namely, no edge is an ``arrow''), then $G_{\Gamma,\lambda}$ is the pro-$\ell$ completion of the discrete RAAG associated to the na\"ive graph $\ddot\Gamma$.
		In particular, if $\Gamma$ has no edges, then $G_{\Gamma,\lambda}$ is the free pro-$\ell$ group generated by $\euV$.
	\end{exam}

	\begin{exam}\label{ex:easy RAAG}\rm
		Let $\Gamma=(\euV,\euE)$ be the specially oriented graph with geometric realization
		
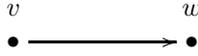
\[
		\xymatrix@R=1.5pt{v && w \\ \bullet \ar[rr] && \bullet}
		\]
		Then 
		$$G=G_{\Gamma,\lambda}=\left\langle\: v,w\:\mid\: {}^wv=v^{\lambda(1)}\:\right\rangle,$$
		and one has $\theta_{\Gamma,\lambda}(v)=1$ and $\theta_{\Gamma,\lambda}(w)=\lambda(1)$.
		Observe that $(G_{\Gamma,\lambda},\theta_{\Gamma,\lambda})$ is $\theta_{\Gamma,\lambda}$-abelian.
	\end{exam}
	
	\begin{rem}\label{rem:titsalt RAAGs}\rm
		A $2$-generated pro-$\ell$ group $G$ is isomorphic to an oriented pro-$\ell$ RAAG associated to an oriented graph (with two vertices) if, and only if, either of the following occurs: $G$ is a free pro-$\ell$ group (in which case, the associated oriented graph consists of two disjoint vertices); or $G$ is locally uniform (in which case, the two vertices of the associated oriented graph are joined).
	\end{rem}
	
	\begin{exam}\label{ex:complete RAAG}\rm
		Let $\Gamma=(\euV,\euE)$ be a complete specially oriented graph and let $\lambda:\Z_p\to\Z_p^\times$ be a linear orientation.
		Then $\Gamma$ has at most one special vertex (cf. Example~\ref{exam:complete specially oriented graph}), say $w$, and 
		$$G_{\Gamma,\lambda}=\left\langle\: w,\euV_\mathrm{o}\:\mid\: [v,v']=1,\:{}^wv=v^{\lambda(1)}\;\forall\:v,v'\in\euV_\mathrm{o}\:\right\rangle$$
		(where we set implicitly $w=1$ and we omit the relations ${}^wv=v^{\lambda(1)}$, if $\euV_\mathrm{s}=\varnothing$),
		which is locally uniform, with canonical orientation $\eth_{G}=\theta_{\Gamma,\lambda}$ --- in particular,
		$\theta_{\Gamma,\lambda}(w)=\lambda(1)$, if $\euV_\mathrm{s}\neq\varnothing$, and $\theta_{\Gamma,\lambda}(v)=1$ for every $v\in\euV_\mathrm{o}$.
	\end{exam}
	
	\begin{exam}\rm
		Let $\Gamma=(\euV,\euE)$ be the oriented graph with geometric realization
		
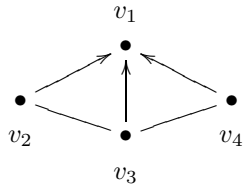
\[
		\xymatrix@R=1.5pt{ & v_1 &\\ & \bullet & \\ && \\ 
			\bullet\ar@{-}[rd]\ar[uur] && \bullet\ar@{-}[ld]\ar[uul]\\ v_2 & \bullet\ar[uuu] & v_4 \\ & v_3&}
		\]
		--- observe that $\Gamma$ is specially oriented, but not of elementary type by Remark~\ref{rem:ETgraph conn}--(b), as the induced subgraph $\Gamma'$ with vertices $v_1,v_2,v_4$ is as the graph $\Lambda_{\mathrm{s}}$ in Example~\ref{exam:ETgraphs}--(c), and $\Gamma=\nabla(\Gamma')$.
		Moreover, $\Gamma$ is the patching of the induced vertices $\Gamma_1,\Gamma_2$ --- with vertices respectively $v_1,v_2,v_3$ and $v_1,v_3,v_4$ --- along the induced subgraph $\Delta$ with vertices $v_1,v_3$.
		Then 
		\[
		\begin{split}
			G_{\Gamma,\lambda} &= \left\langle\:v_1,\ldots,v_4\:\mid\:{}^{v_1}v_i=v_i^{\lambda(1)},\:[v_2,v_3]=[v_3,v_4]=1,\:i=2,3,4\:\right\rangle\\
			&\simeq G_{\Gamma_1,\lambda}\amalg_{G_{\Delta,\lambda}}^{\hat\ell} G_{\Gamma_2,\lambda},
		\end{split}
		\]
		where $G_{\Gamma_1,\lambda}$, $G_{\Gamma_2,\lambda}$ and $G_{\Delta,\lambda}$ are all locally uniform.
		(The oriented graph $\Gamma$ is {\sl chordal}, cf. \S~\ref{sec:chord} below.)
	\end{exam}

	Unlike (pro-$\ell$ completions of) RAAGs, an oriented pro-$\ell$ RAAG
	may yield non-trivial torsion, as shown by the following example.

	\begin{exam}\label{ex:mennike}\rm
		For $\ell$ odd, set $\lambda\colon\Z_{\ell}\to\Z_{\ell}^\times$ such that $\lambda(1)=1+{\ell}$, and let $\Gamma=(\euV,\euE)$ be the oriented graph
		
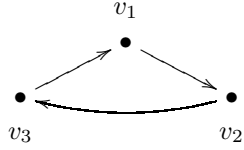
\[
		\xymatrix@R=1.5pt{ & v_1 &\\ & \bullet\ar[rdd] & \\ && \\ 
			\bullet\ar[ruu] && \bullet\ar@/^/[ll]\\ v_3 & & v_2}
		\]
		Then one has
		\[	G_{\Gamma,\lambda}=\left\langle\: v_1,v_2,v_3\:|\:{}^{v_2}v_1=v_1^{1+\ell},\:{}^{v_3}v_2=v_2^{1+\ell},\:
		{}^{v_1}v_3=v_3^{1+\ell}\:\right\rangle,
		\]
		and this pro-$\ell$ group is a finite $\ell$-group, as shown by J.~Mennicke (cf. \cite[Ch.~I, \S~4.4, Ex.~2(e)]{serre:gal}).
		Observe that on the one hand the oriented graph $\Gamma$ is not specially oriented; on the other hand, the orientation $\theta_{\Gamma,\lambda}$ is constantly equal to 1 as $\euV=\euV_{\mathrm{o}}$, and the oriented pro-$\ell$ group $(G_{\Gamma,\lambda},\theta_{\Gamma,\lambda})$ is not Kummerian, as $\kernel(\theta_{\Gamma,\lambda})/K_{\theta_{\Gamma,\lambda}}(G_{\Gamma,\lambda})=G_{\Gamma,\theta}/G_{\Gamma,\lambda}'\simeq (\Z/\ell\Z)^3$.
	\end{exam}


	\subsection{Specially oriented graphs and Kummerianity}\label{ssec:cyc label}
	
	The goal of this subsection is to prove that Kummerianity characterises oriented pro-$\ell$ RAAGs associated to specially oriented graphs.
	
	Let $\Gamma=(\euV,\euE)$ be an oriented graph, let $\lambda\colon \Z_{\ell}\to\Z_{\ell}^\times$ be a linear orientation, and let $(G_{\Gamma,\lambda},\theta_{\Gamma,\lambda})$ be the associated oriented pro-$\ell$ RAAG.
	First, we need the following lemma.
	
	\begin{lem}\label{lem:kumm orRAAG}
		Let $\theta\colon G_{\Gamma,\lambda}\to \Z_\ell^\times$ be a torsion-free orientation such that the oriented pro-$\ell$ group $(G_{\Gamma,\lambda},\theta)$ is Kummerian, and let $\varphi\colon G_{\Gamma,\lambda}\to \bar G$ be an epimorphism satisfying condition--(iii) of Proposition~\ref{prop:kummer}.
		Moreover, for every edge $\eue\in\euE$, let $H_{\eue}$ be the subgroup of $G_{\Gamma,\lambda}$ generared by $o(\eue),t(\eue)$.
		Then the restriction $$\varphi\vert_{H_{\eue}}\colon H_{\eue}\longrightarrow \bar G$$
		is a monomorphism of locally uniform pro-$\ell$ groups.
	\end{lem}
	
	\begin{proof}
		Set $G=G_{\Gamma,\lambda}$.
		By Proposition~\ref{prop:kummer}, $\kernel{\varphi}=K_\theta(G)$, and thus $\eth_{\bar G}=\theta_{/\kernel(\theta)}$. 
		For every edge $e=(v,w)\in\euE$, $\varphi(H_{\eue})$ is a subgroup of the locally uniform pro-$\ell$ group $\bar G$, and thus it is locally uniform as well, with associated canonical orientation $\eth_{\varphi(H_{\eue})}=\eth_{\bar G}\vert_{\varphi(H_{\eue})}$.
		Moreover, $K_\theta(G)\subseteq \Phi(G)$ by definition, so that $\varphi(v)$ and $\varphi(w)$ are linearly independent modulo $\Phi(\bar G)$, and $H_{\eue}$ is 2-generated.
		
		On the other hand, consider the epimorphism of pro-$\ell$ groups
		$$  \varphi\vert_{H_{\eue}}\colon H_{\eue}\longrightarrow\varphi(H_{\eue}).
		$$
		By Example~\ref{ex:easy RAAG}--(a), $H_{\eue}$ is a quotient of the 2-generated locally uniform pro-$\ell$ group 
		$$H=\langle\:v,w\:\mid\:R_{\lambda}(\eue)=1\:\rangle,$$
		with associated canonical orientation $\eth_{H}$.
		Altogether, one has a chain of epimorphisms of pro-$\ell$ groups 
		$H\to H_{\eue}\to \varphi(H_{\eue})$, where both $H$ and $\varphi(H_{\eue})$ are 2-generated locally uniform pro-$\ell$ groups.
		Therefore, both homomorphism are isomorphisms of 2-generated locally uniform pro-$\ell$ groups, and one has a chain of isomorphism of oriented pro-$\ell$ groups
		$$\xymatrix{ (H,\eth_H)\ar[r] & (H_{\eue},\theta\vert_{H_\eue})\ar[r]^-{\varphi} & \left(\varphi(H_{\eue}),\eth_{\bar G}\vert_{\varphi(H_{\eue})}\right) },$$
		as the canonical orientations depend uniquely by the structures of the locally uniform pro-$\ell$ groups by Proposition~\ref{prop:locunif}. 
		In particular, one has 
		\begin{equation}\label{eq:monomorphism kumm RAAG 2}
			\eth_H(v)=\eth_{H_{\eue}}(v)=\eth_{\bar G}(\varphi(v))=\theta(v)   
		\end{equation}
		--- where the last equality follows by the fact that $\eth_{\bar G}\circ\varphi=\theta$ by Proposition~\ref{prop:kummer} ---, and analogously for $w$.
	\end{proof}

	\begin{thm}\label{thm:cyc pRAAG}
		Let $\Gamma=(\euV,\euE)$ be an oriented graph, let $\lambda\colon \Z_{\ell}\to\Z_{\ell}^\times$ be a linear orientation, 
		and let $(G_{\Gamma,\lambda},\theta_{\Gamma,\lambda})$ be the associated oriented pro-$\ell$ RAAG.
		Then there exists a torsion-free orientation $\theta\colon G\to\Z_{\ell}^\times$ such that the oriented pro-$\ell$ group
		$(G_{\Gamma,\theta},\theta)$ is Kummerian if, and only if, $\Gamma$ is specially oriented and $\theta=\theta_{\Gamma,\lambda}$.
	\end{thm}
	
	\begin{proof}
		Set $G=G_{\Gamma,\lambda}$.
		
		Suppose first that $\Gamma$ is specially oriented.
		Let $F$ be the free pro-$\ell$ group generated by $\euV$, and let $R_{\Gamma,\lambda}$ be the normal subgroup of $F$ generated (as a normal subgroup) by $\{R_{\lambda}(\eue),\eue\in\euE\}$ --- namely, one has a short exact sequence of pro-$\ell$ groups
		$$\xymatrix{\{1\}\ar[r] & R_{\Gamma,\lambda}\ar[r] & F\ar[r]^-{\pi} & G\ar[r] & \{1\}}.$$
		Set $\tilde\theta=\theta_{\Gamma,\lambda}\circ\pi$, and consider the oriented pro-$\ell$ group and $(F,\tilde\theta)$.
		By Example~\ref{exam:kummer groups}--(b), $(F,\tilde\theta)$ is Kummerian, and thus $F/K_{\tilde\theta}(F)$ is a locally uniform pro-$\ell$ group by Proposition~\ref{prop:kummer}.
		We claim that $R_{\Gamma,\lambda}\subseteq K_{\tilde\theta}(F)$.
		Indeed, for every edge $\eue=(x,w)\in\euE$, one has the following.
		\begin{itemize}
			\item[(i)] If $\eue\in\euE_\mathrm{s}$, then $x\in\euV_\mathrm{s}$, so that $y\in\euV_\mathrm{o}$ by condition~(ii) in Definition~\ref{defi:specialgraphs}.
			Consequently, $\theta_{\Gamma,\lambda}(x)=\lambda(1)$ and $y\in\kernel(\theta_{\Gamma,\lambda})$, so that 
			\[
			R_{\lambda}({\eue})=[x,y]y^{1-\lambda(1)}=[x,y]y^{1-\theta_{\Gamma,\lambda}(x)}\in K_{\tilde\theta}(F).
			\]
			\item[(ii)] If $\eue\in\euE_\mathrm{o}$, then $x,y\in\euV_\mathrm{o}$ by condition~(i) in Definition~\ref{defi:specialgraphs}.
			Consequently, $x,y\in\kernel(\theta_{\Gamma,\lambda})$, so that 
			\[
			R_{\lambda}({\eue})=[x,y]\in K_{\tilde\theta}(F).
			\]
		\end{itemize}
		Therefore, 
		\[
		\frac{G}{K_{\theta_{\Gamma,\lambda}}(G)}\simeq \frac{F/R_{\Gamma,\lambda}}{K_{\tilde\theta}(F)/R_{\Gamma,\lambda}}
		\simeq\frac{F}{K_{\tilde\theta}(F)},
		\]
		where the latter is a locally uniform pro-$\ell$ group.
		Hence, $(G,\theta_{\Gamma,\lambda})$ is Kummerian by Proposition~\ref{prop:kummer}.

		Suppose now that $\Gamma$ is not specially oriented, and that $\theta\colon G\to\Z_\ell^\times$ is a torsion-free orientation such that $(G,\theta)$ is a Kummerian oriented pro-$\ell$ group.
		Then $\Gamma$ contains a (non-necessarily induced) subgraph $\Gamma'=(\euV',\euE')$, with $\euV'=\{x,y,z\}$, as in \eqref{eq:special bad}.
		Set $\eue_1=(x,y)\in\euE$ and $\eue_2=(z,x)\in\euE_\mathrm{s}$.
		By Lemma \ref{lem:kumm orRAAG}, one has
		$$H_{\eue_2}=\left\langle\:x,z\:\mid\:{}^xz=z^{\lambda(1)}\:\right\rangle,$$
		and $\eth_{H_{\eue_2}}(x)=\theta(x)=\lambda(1)$ by \eqref{eq:monomorphism kumm RAAG 2}.
		Now one has two cases. 
		\begin{itemize}
			\item[(i)] If $\Gamma'$ is as the left-side one in \eqref{eq:special bad} --- i.e., $\eue_1\in\euE_\mathrm{o}$ ---, then $[x,y]=1$ and $H_{\eue_1}\simeq\Z_\ell^2$ by Lemma~\ref{lem:kumm orRAAG}.
			In particular,
			$$\theta(x)=\eth_{H_{\eue_1}}(x)=1\neq\lambda(1)$$
			by \eqref{eq:monomorphism kumm RAAG 2}, a contradiction.
			\item[(ii)] If $\Gamma'$ is as the right-side one in \eqref{eq:special bad} --- i.e., $\eue_1\in\euE_\mathrm{s}$ ---, one has $$H_{\eue_1}\simeq\left\langle\: x,y\:\mid\:{}^yx=x^{\lambda(1)}\:\right\rangle$$ by Lemma~\ref{lem:kumm orRAAG}.
			In particular, $\theta(x)=\eth_{H_{\eue_1}}(x)=1\neq\lambda(1)$
			by \eqref{eq:monomorphism kumm RAAG 2}, a contradiction.
		\end{itemize}
		Therefore, if $\Gamma$ is not specially oriented, $(G,\theta)$ cannot be Kummerian.
	\end{proof}
	
	From Theorem~\ref{thm:galois} and from Theorem~\ref{thm:cyc pRAAG}, one deduces that the oriented pro-$\ell$ RAAG $(G_{\Gamma,\lambda},\theta_{\Gamma,\lambda})$ associated to an oriented graph $\Gamma$ and a linear orientation $\lambda\colon\Z_\ell\to\Z_\ell^\times$ may occur as the maximal pro-$\ell$ Galois group of a field $\K$ containing a primitive $\ell$-th root of unity (and also $\sqrt{-1}$ if $\ell=2$) only if $\Gamma$ is specially oriented, and with $\theta_{\Gamma,\lambda}=\hat\theta_{\K}$.
	This is a refinement of \cite[Thm.~5.32]{qsv:quadratic}.

	\subsection{Cliques and locally uniform oriented pro-$\ell$ RAAGs}
	From Theorem~\ref{thm:cyc pRAAG} we deduce the following.
	
	\begin{cor}\label{cor:complete RAAG cyc}
		Let $\Gamma=(\euV,\euE)$ be an oriented graph, and let $\lambda\colon \Z_{\ell}\to\Z_{\ell}^\times$ be a linear orientation.
		The associated oriented pro-$\ell$ RAAG $(G_{\Gamma,\lambda},\theta_{\Gamma,\lambda})$ is locally uniform if, and only if, $\Gamma$ is complete and specially oriented.
	\end{cor}
	
	\begin{proof}
		Set $G=G_{\Gamma,\lambda}$.
		Suppose that $\Gamma$ is complete and specially oriented.
		If $\euV_\mathrm{s}=\varnothing$, then $G$ is the pro-$\ell$ completion of the discrete RAAG associated to the na\"ive graph $\ddot\Gamma$, and thus it is a free abelian pro-$\ell$ group.
		Otherwise, there is only one special vertex $w\in\euV_\mathrm{s}$ (cf. Example~\ref{exam:complete specially oriented graph}), and thus
		\[
		G=\left\langle\:\euV\:\mid\:[v,v']=1,\:{}^wv=v^{\lambda(1)}\;\forall\:v,v'\in\euV_\mathrm{o}\:\right\rangle
		\]
		is locally uniform.
		
		Conversely, suppose that $G$ is locally uniform, with canonical orientation $\eth_G$.
		Then $(G,\eth_G)$ is a Kummerian oriented pro-$\ell$ group, and therefore by Theorem~\ref{thm:cyc pRAAG}, $\Gamma$ is specially oriented and $\eth_G\equiv\theta_{\Gamma,\lambda}$.
		In particular, given two vertices $v,v'\in \euV$, let $H$ be the subgroup of $G$ generated by $v,v'$.
		Then $H$ is locally uniform with $\eth_H\equiv\theta_{\Gamma,\lambda}\vert_H$ (cf. Lemma~\ref{lem:kumm orRAAG}),
		so that $v,v'$ are joined by an edge. 
	\end{proof}

	By the proof of Theorem~\ref{thm:cyc pRAAG}, the subgroup of an oriented pro-$\ell$ RAAG, associated to a specially oriented graph $\Gamma$ and a linear orientation, generated by two adjacent vertices of $\Gamma $ is isomorphic to the oriented pro-$\ell$ RAAG associated to the induced subgraph of $\Gamma$ with these two vertices.
	This is true also for every clique of $\Gamma$.
	
	\begin{prop}\label{prop:specialRAAG inclusion}
		Let $\Gamma=(\euV,\euE)$ be a specially oriented graph and $\lambda\colon \Z_{\ell}\to\Z_{\ell}^\times$ a linear orientation.
		If $\Delta=(\euV(\Delta),\euE(\Delta))$ is a clique of $\Gamma$, then the inclusion $\euV(\Delta)\hookrightarrow\euV$ induces a monomorphism of oriented pro-$\ell$ groups 
		$$(G_{\Delta,\theta},\theta_{\Delta,\lambda})\longrightarrow (G_{\Gamma,\lambda},\theta_{\Gamma,\lambda}).$$
	\end{prop}
	
	\begin{proof}
		By Corollary~\ref{cor:complete RAAG cyc}, $(G_{\Delta,\lambda},\theta_{\Delta,\lambda})$ is $\theta_{\Delta,\lambda}$-abelian --- and thus the pro-$\ell$ group $G_{\Delta,\theta}$ is locally uniform ---, as $\Delta$ is complete and specially oriented.
		
		Set $G=G_{\Gamma,\lambda}$.
		Since $\Gamma$ is specially oriented, $(G,\theta_{\Gamma,\lambda})$ is a Kummerian oriented pro-$\ell$ group by Theorem~\ref{thm:cyc pRAAG}, and thus $G/K_{\theta_{\Gamma,\lambda}}(G)$ is a locally uniform pro-$\ell$ group by Proposition~\ref{prop:kummer}.
		Let
		\begin{equation}\label{eq:phi emb clique}
			\phi_\Delta\colon G_{\Delta,\theta}\longrightarrow G\qquad\text{and}\qquad
			\varphi\colon G\longrightarrow G/K(\calG_{\Gamma})
		\end{equation}
		denote respectively the morphism induced by the inclusion $\euV(\Delta)\hookrightarrow\euV$ and the canonical projection --- note that both maps $\pi\vert_\euV$ and $(\pi\circ\phi_{\Delta})\vert_{\euV(\Delta)}$ are injective, as $K_{\theta_{\Gamma,\lambda}}(G)\subseteq \Phi(G)$ --- and let $H$ be the subgroup of $G$ generated by $\euV(\Delta)$.
		Then one has a chain of morphism of pro-$\ell$ groups
		\[
		\xymatrix{ G_{\Delta,\lambda}\ar[r]^-{\phi_\Delta} & H\ar[r]^-{\pi\vert_H} & \pi(H) },
		\]
		where both $G_{\Delta,\lambda}$ and $\pi(H)$ are locally uniform pro-$\ell$ groups, minimally generated by $\euV(\Delta)$.
		Therefore, both morphisms are isomorphism, and $H$ is locally uniform as well.
		In particular, $\theta_{\Delta,\lambda}=\eth_{\pi(H)}\circ\pi\circ\phi_\Delta$, as the canonical orientations are uniquely determined by the structure of the locally uniform pro-$\ell$ groups (cf. Proposition~\ref{prop:locunif}).
	\end{proof}
	
	
	\subsection{Cohomology of oriented pro-$\ell$ RAAGs}\label{ssec:RAAG cohom}
	
	Let $\ddot\Gamma=(\ddot\euV,\ddot\euE)$ be a na\"ive graph, and let $V$ denote the $\F_\ell$ vector space with basis
	$\ddot\euV$.
	The {\sl exterior Stanley-Reisner $\F_\ell$-algebra} associated to $\ddot\Gamma^{\op}$ is the $\F_{\ell}$-algebra
	\[
	\bfLam_\bullet(\ddot\Gamma^{\op})=\frac{\bfLam_\bullet(V)}{\left(\:v\wedge w\:\mid\: \{v,w\}\notin \ddot\euE\:\right)},
	\]
	where $\bfLam_\bullet(V)$ denotes the exterior $\F_\ell$-algebra generated by $V$.
	Clearly, $\bfLam_\bullet(\ddot\Gamma^{\op})$ is a quadratic algebra.
	
	Now let $\Gamma=(\euV,\euE)$ be an oriented graph, and let $\lambda$ be a linear orientation.
	By duality, one has an isomorphism of $\F_\ell$-vector spaces
	\begin{equation}\label{eq:H1 RAAG}
		V=\bfLam_1(\ddot\Gamma^{\op}) \overset{\sim}{\longrightarrow}
		\rmH^1(G_{\Gamma,\lambda},\F_\ell)=\mathrm{Hom}(G_{\Gamma,\lambda},\F_\ell).
	\end{equation}
	(cf. \cite[Ch.~I, \S~4.2]{serre:gal}).
	By \cite[Lemma~5.8]{qsv:quadratic}, the cup-product extends the isomorphism \eqref{eq:H1 RAAG} to a homomorphism of graded $\F_\ell$ algebras 
	\begin{equation}\label{eq:H raag}
		\bfLam_\bullet(\ddot\Gamma^{\op})\longrightarrow\bfH^\bullet(G_{\Gamma,\lambda},\F_\ell) 
	\end{equation}
	which is an isomorphism in degree 2 too.
	Therefore, if $\bfH^\bullet(G_{\Gamma,\lambda},\F_{\ell})$ is quadratic, then $\bfH^\bullet(G_{\Gamma,\lambda},\F_{\ell})\simeq\bfLam_\bullet(\ddot\Gamma^{\op})$ (cf. \cite[Thm.~E]{qsv:quadratic}).
	
	\begin{exam}\rm
		Let $\lambda$ be a linear orientation, and let $\Gamma=(\euV,\euE)$ be an oriented graph without triangles as induced subgraphs.
		Then $\bfH^\bullet(G_{\Gamma,\lambda},\F_{\ell})$ is isomorphic to $\bfLam_\bullet(\ddot\Gamma^{\op})$ (cf. \cite[Thm.~F]{qsv:quadratic}).
	\end{exam}

	\begin{rem}\label{rem:bockstein}\rm
		Put $\ell=2$, and let $\lambda\colon \Z_2\to\Z_2^\times$ be a linear orientation --- i.e., $\lambda(1)=1+4\mu$ for some $\mu\in\Z_2\smallsetminus\{0\}$.
		Given an oriented graph $\Gamma=(\euV,\euE)$, set $G=G_{\Gamma,\lambda}$.
		Then the quotient $G/G^4G'$ is isomorphic to $(\Z/4\Z)^d$, where  $d$ is the number of vertices of $\Gamma$, so that the map
		\[
		\rmH^1(G_{\Gamma,\lambda},\Z/4\Z)\longrightarrow\rmH^1(G_{\Gamma,\lambda},\F_2),
		\]
		induced by the canonical projection $\Z/4\Z\to\F_2$, is surjective.
		This implies that the {\sl Bockstein morphism} $\mathfrak{b}_G\colon \rmH^1(G,\F_2)\to\rmH^2(G,\F_2)$
		is trivial (cf., e.g., \cite[p.~415]{cq:1smoothBK}).
		Furthermore, if $H$ is a subgroup of $G$ such that the restriction map
		$\mathrm{res}_{G,H}^1\colon \rmH^1(G,\F_\ell)\to\rmH^1(H,\F_\ell)$ is surjective, the commutativity of the diagram
		\[
		\xymatrix{ \rmH^1(G,\F_\ell)\ar[rr]^-{\mathfrak{b}_G}\ar[d]_-{\mathrm{res}_{G,H}^1} && 
			\rmH^2(G,\F_\ell)\ar[d]^-{\mathrm{res}_{G,H}^1} \\
			\rmH^1(H,\F_\ell)\ar[rr]^-{\mathfrak{b}_H} && \rmH^2(H,\F_\ell)}
		\]
		implies that also the Bockstein morphism $\mathfrak{b}_H\colon \rmH^1(H,\F_2)\to\rmH^2(H,,\F_2)$
		is trivial.
	\end{rem}
	


	\section{Bloch-Kato Oriented pro-$\ell$ RAAGs}
	
	\subsection{A Tits' alternative}\label{ssec:titsalt}
	
	Recall that by Remark~\ref{rem:locunif cohom} a locally uniform pro-$\ell$ group is Bloch-Kato.
	For Bloch-Kato pro-$\ell$ groups one has the following Tits' alternative type result (cf. \cite[\S~7.1]{qw:cyclotomic}).
	
	\begin{prop}\label{prop:titsalt}
		Let $G$ be a Bloch-Kato pro-$\ell$ group.
		If $\ell=2$ suppose further that the Bockstein morphism $\mathfrak{b}_G\colon\rmH^1(G,\F_2)\to\rmH^2(G,\F_2)$ is trivial.
		Then either $G$ is locally uniform, or $G$ contains a subgroup which is a free non-abelian pro-$\ell$ group.
		In particular, if $G$ is 2-generated, then either $G$ is a free pro-$\ell$ group; or $G$ is locally uniform.
	\end{prop}
	
	By Remark~\ref{rem:titsalt RAAGs}, the same phenomenon occurs for oriented pro-$\ell$ RAAGs. Moreover, observe that by Remark~\ref{rem:bockstein}, if $\ell=2$ the Bockstein morphism is always trivial for every subgroup of an oriented pro-$\ell$ RAAG such that the restriction map of degree 1 is surjective.
	
	
	\subsection{Non-specially oriented graphs}

	\begin{thm}\label{thm:nospecial noBK}
		Let $\Gamma=(\euV,\euE)$ be an oriented graph, and let $\lambda$ be a linear orientation.
		If $G_{\Gamma,\lambda}$ is Bloch-Kato, then $\Gamma$ is specially oriented.
	\end{thm}
	
	\begin{proof}
		Set $G=G_{\Gamma,\lambda}$, and suppose that $\Gamma$ is not specially oriented.
		Then $\Gamma$ contains a (non-necessarily induced) subgraph with geometric realization as in \eqref{eq:special bad} (cf. Definition~\ref{defi:specialgraphs}).
		Set $\eue_1=(x,y),\eue_2=(z,x)\in\euE$.
		Let $H$ be the subgroup of $G$ generated by $x,y,z$, and let $H_0$ be the subgroup of $H$ generated by $y,z$.
		If $H_0$ is not free nor locally uniform, then by Proposition~\ref{prop:titsalt} $H_0$ (and thus also $G$) is not Bloch-Kato.
		So, let us assume that $H_0$ is locally uniform or free.
		
		Observe that the subgroups $H_{\eue_1}$ and $H_{\eue_2}$ of $G$, generated respectively by $x,y$ and $x,z$, are not free.
		Therefore, if $G$ is Bloch-Kato, then they must be locally uniform by Proposition~\ref{prop:titsalt}.
		
		Suppose first that $H_0$ is locally uniform: then $[y,z]\in H_0^{\ell}\subseteq H^\ell$, and thus $H$ is powerful, as also $[x,y],[x,z]\in H^\ell$.
		If $H$ is locally uniform, then one has 
		$$
		\eth_H(x)=\eth_{H_{\eue_1}}(x)=1\qquad\text{and}\qquad \eth_H(x)=\eth_{H_{\eue_2}}(x)=\lambda(1),$$
		a contradiction.
		Therefore, $H$ is powerful but not locally uniform, and since powerful pro-$\ell$ groups contain no non-abelian free subgroups (cf. \cite[Thm.~3.13]{ddsms}), $H$ is not Bloch-Kato by Proposition~\ref{prop:titsalt}.
		
		Suppose now that $H_0$ is a free pro-$\ell$ group, and set $t=yz$.
		Then
		\begin{equation}\label{eq:Ht 2gen nospec}
			{}^xt={}^xy\cdot {}^xz=\begin{cases}
				y\cdot {}^xz=yz^{\lambda(1)}=tz^{\lambda(1)-1},&\text{if }(x,y)\in\euE_\mathrm{o},\\
				x^{1-\lambda(1)}yz^{\lambda(1)}=x^{1-\lambda(1)}tz^{\lambda(1)-1},&\text{if }(x,y)\in\euE_\mathrm{s}.
			\end{cases}
		\end{equation}
		Let $H_t$ be the subgroup of $H$ generated by $x,t$.
		Observe that \eqref{eq:Ht 2gen nospec} implies that $z^{\lambda(1)-1}$ lies in $H_t$, and therefore $H_t$ is not a free pro-$\ell$ group, as 
		$${}^x\left(z^{\lambda(1)-1}\right)=z^{\lambda(1)(\lambda(1)-1)}.$$
		Moreover, Proposition~\ref{prop:titsalt} may apply also to $H_t$ (cf. \S~\ref{ssec:titsalt}).
		If $H_t$ is locally uniform, then also the subgroup $\langle\: t,z^{\lambda(1)-1}\:\rangle\subseteq H_t$ is locally uniform, but $\langle\: t,z^{\lambda(1)-1}\:\rangle$ is also a subgroup of $H_0$ which is free: therefore, 
		$$\langle\: t,z^{\lambda(1)-1}\:\rangle\simeq\Z_{\ell}\;\Longrightarrow\;t^{\mu_1}=z^{\mu_2(\lambda(1)-1)}
		\qquad\text{ for some }\mu_1,\mu_2\in\Z_{\ell}\smallsetminus\{0\}.$$
		This implies that $[z,t^{\mu_1}]=1$, while $H_0=\langle\: t,z\:\rangle$ is free, a contradiction.
		Therefore, $H_t$ is not free nor locally uniform, and thus Proposition~\ref{prop:titsalt} implies that $H_0$ (and hence also $G$) is not Bloch-Kato.
	\end{proof}
	
	From the proof of Theorem~\ref{thm:nospecial noBK} we deduce the following.
	
	\begin{cor}\label{cor:nospecial nosubgroups}
		Let $\Gamma=(\euV,\euE)$ be an oriented graph, and let $\lambda$ be a linear orientation.
		If every finitely generated subgroup of $G_{\Gamma,\lambda}$ is isomorphic to an oriented pro-$\ell$ RAAG, then $\Gamma$ is specially oriented. 
	\end{cor}
	
	\begin{proof}
		Set $G=G_{\Gamma,\lambda}$. 
		Suppose that $\Gamma$ is not specially oriented, and let $x,y,z\in\euV$ and $H,H_0,H_t\subseteq G$ be as in the proof of Theorem~\ref{thm:nospecial noBK}.
		
		If the 2-generated pro-$\ell$ group $H_0$ is not a free pro-$\ell$ group nor locally uniform, then $H_0$ does not occur as an oriented pro-$\ell$ RAAG by Remark~\ref{rem:titsalt RAAGs}.
		So, let us suppose that $H_0$ is free or locally uniform.
		
		If $H_0$ is a free pro-$\ell$ group, then the 2-generated pro-$\ell$ group $H_t$ is not free nor locally uniform, and hence $H_t$ does not occur as an oriented pro-$\ell$ RAAG, again by Remark~\ref{rem:titsalt RAAGs}.
		
		Finally, suppose that $H_0$ is locally uniform.
		Then $H$ is powerful --- but not locally uniform, by the proof of Theorem~\ref{thm:nospecial noBK} ---, and by \cite[Thm.~3.13]{ddsms} $H$ has no free 2-generated subgroups.
		Suppose that every 2-generated subgroup of $H$ is locally uniform. 
		Thus, $H$ is torsion-free.
		Moreover, for every couple of elements $u,w\in H$, the subgroup $\langle\:u,w\:\rangle$ is uniform, and 
		$$[u,w]\in\begin{cases} \langle\:u,w\:\rangle^{\ell} & \text{if }\ell\neq2, \\ 
			\langle\:u,w\:\rangle^4&\text{if }{\ell}=2.  \end{cases}$$
		Therefore, for every subgroup $V$ of $H$ one has $V'\subseteq V^{\ell}$ ($V'\subseteq V^4$ if ${\ell}=2$), so that $V$ is powerful, and also uniform as $H$ is torsion-free.
		Hence, $H$ is locally uniform, a contradiction.
		Thus, there exists a 2-generated subgroup of $H$ which is neither free nor locally uniform, and consequently it does not occur as an oriented pro-$\ell$ RAAG.
	\end{proof}
	
	\begin{rem}\label{rem:2subgroups locunif}\rm
		From the proof of Corollary~\ref{cor:nospecial nosubgroups}, one deduces that a pro-$\ell$ group $G$ is locally uniform if, and only if, every 2-generated subgroup of $G$ is uniform.
	\end{rem}

	
	\subsection{The oriented graph $\Lambda_{\mathrm{s}}$}
	
	Recall that the oriented graph $\Lambda_{\mathrm{s}}$ is specially oriented but not of elementary type (cf. Example~\ref{exam:ETgraphs}--(c)).
	
	\begin{prop}\label{thm:lambdas}
		Let $\lambda$ be a linear orientation.
		Then the pro-$\ell$ group $G_{\Lambda_{\mathrm{s}},\lambda}$ is not Bloch-Kato.
	\end{prop}

	\begin{proof}
		Set $G=G_{\Lambda_{\mathrm{s}},\lambda}$.
		Up to a change of the generator $x$, we may assume that 
		$$\theta_{\Lambda_{\mathrm{s}},\lambda}(x)=\lambda(1)=1+p^f\qquad \text{for some }f\geq1$$
		(in particular, if ${\ell}=2$ then $f\geq2$). Put $q=p^f$.
		Then $G$ has a minimal presentation
		\[
		G=\left\langle\:x,z_1,z_2\:\mid\:{}^xz_1=z_1^{1+q},\:{}^xz_2=z_2^{1+q}\:\right\rangle.
		\]
		Hence, we can consider $G$ as the fundamental group of the graph of pro-$\ell$ groups
		\[
		\xymatrix@R=1.5pt{ \bullet\ar@{-}[rr]^-{Z} && \bullet\\ G_1 & & G_2}
		\]
		where $Z$ is the pro-$\ell$-cyclic subgroup of $G$ generated by $x$, and for $i=1,2$ the group $G_i$ is the subgroup of $G$ generated by $x,z_i$ --- i.e., $G_i\simeq\langle z_i\rangle\rtimes Z$ ---, and the monomorphisms from $Z$ to $G_1,G_2$ are the canonical embeddings.
		
		Let $\phi_U\colon G\to\Z/q\Z$ be the homomorphism defined by $\phi_U(z_i)=1$ for $i=1,2$, and $\phi_U(x)=0$.
		Set $u_i=z_i^q$, for $i=1,2$, and $t=z_1z_2^{-1}$.
		Then $U$ is the normal subgroup of $G$ generated by $x,u_1,u_2,t$.
		Note that $x,u_1,t$ are enough to generate $U$ (as normal subgroup), as
		\begin{equation}\label{eq: u2 gen by}
			{}^xt=z_1^{1+q}z_2^{-1-q}=u_1 t  u_2.
		\end{equation}
		Let $U_1,U_2$ be the subgroups of $U$ generated by $u_1,x$ and $u_2,x$ respectively.
		By Theorem~\ref{thm:zalmel}, $U$ is the fundamental pro-$\ell$ group of the graph of pro-$\ell$ groups $\mathcal{U}$ based on the graph $\Delta=U\backslash\mathrm{T}$, with $q$ edges, where $\mathrm{T}=(\euV(\mathrm{T}),\euE(\mathrm{T}))$ is the second-countable pro-$\ell$ tree associated to the decomposition as pro-$\ell$ amalgam $G\simeq G_1\amalg_Z^{\hat {\ell}}G_2$, namely 
		$$\euV(\mathrm{T})=\left\{\:gG_1,\:gG_2\:\mid\: g\in G\:\right\}\qquad 
		\text{and}\qquad \euE(\mathrm{T})=\{\:gZ\:\mid\: g\in G\:\}.$$ 
		In particular $\mathcal{U}(v_i)=U_i$ for $i=1,2$.
		
		Let $H$ be the subgroup of $U$ generated as a pro-$\ell$ group by $\{u_1,x,t\}$, and set $y=u_1^{-1}x$.
		Then $u_2\in H$, and $U_1=\langle\: u_1\:\rangle\rtimes\langle \: y\:\rangle$.
		Moreover, $H$ is the fundamental pro-$\ell$ group of the graph of pro-$\ell$ groups
		\[\mathcal{H}=\qquad 
		\xymatrix@R=1.5pt{  \bullet\ar@{-}[rr]_-{Z} \ar@{-}@/^1pc/[rr]^-{\langle\: y\:\rangle} && \bullet\\ U_1&& \langle\: u_2^{-1}x\:\rangle\simeq\Z_{\ell}}
		\]
		which is a restriction of $\mathcal{U}$.
		In particular, the monomorphisms of pro-$\ell$ groups associated to the bottom edge of $\mathcal{U}$ are the inclusion $Z\hookrightarrow U_1$ and the isomorphism $Z\to\langle\: u_2^{-1}x\:\rangle$, $x\mapsto u_2^{-1}x$, while the monomorphisms of pro-$\ell$ groups associated to the top edge of $\mathcal{U}$ are the monomorphism $\langle \:y\:\rangle\to U$, $y\mapsto x$, and the isomorphism $\langle\: y\:\rangle\to\langle\: u_2^{-1}x\:\rangle$, $y\mapsto u_2^{-1}x$.
		Hence 
		\[
		H=\left\langle\:u_1,y,t,u_2^{-1}x\:\mid\:{}^yu_1=u_1^{1+q},\:t^{-1}yt=u_2^{-1}x\:\right\rangle,
		\]
		and this yields the minimal presentation
		\begin{equation}\label{eq:pres H BK}
			H=\left\langle\:u_1,y,t\:\mid\:[y,u_1]u_1^{-q}=
			\left[u_1y,[t^{-1},y]\right]u_2^q\left({}^{(u_2^{-1})}u_1^{-1}\right)^q=1\:\right\rangle
		\end{equation}
		--- observe that $u_2$ may be generated by $u_1,y,t$, as $u_2=[t^{-1},y]u_1^{-1}$ by \eqref{eq: u2 gen by}.
		
		Now let $F$ be the free pro-$\ell$ group generated by $\{u_1,y,t\}$, and let $R$ be the normal subgroup of $F$ such that $H=F/R$.
		Then $R/R^{\ell}[R,F]\neq RF_{(3)}/F_{(3)}$, because 
		$$[u_1,y]u_1^{-q}\notin F_{(3)}\qquad\text{and}\qquad
		\left[u_1y,[t^{-1},y]\right]u_2^q\left({}^{(u_2^{-1})}u_1^{-1}\right)^q\in F_{(3)}$$ 
		(recall that $4\mid q$ if ${\ell}=2$).
		Hence, by \cite[Thm.~7.3]{MPQT}, the cup-product
		\begin{equation}\label{eq:cupproduct Lambdas}
			\rmH^1(H,\F_{\ell})\times \rmH^1(H,\F_{\ell})\overset{\smallsmile}{\longrightarrow} \rmH^2(H,\F_{\ell})
		\end{equation}
		is not surjective, and $\bfH^\bullet(H,\F_{\ell})$ is not a quadratic algebra (see also \cite[Prop.~2.4]{qsv:quadratic}).
	\end{proof}

	\begin{cor}\label{cor:lambdasRAAG}
		Let $\lambda$ be a linear orientation.
		Then the pro-$\ell$ group $G_{\Lambda_{\mathrm{s}},\lambda}$ contains a finitely generated subgroup which does not occur as an oriented pro-$\ell$ RAAG.
	\end{cor}
	
	\begin{proof}
		Let $G$ and $H$ be as in the proof of Proposition~\ref{thm:lambdas}.
		Since the cup-product \eqref{eq:cupproduct Lambdas} is not surjective, $\rmH^2(H,\F_\ell)$ is not isomorphic to $\Lambda_2(\ddot\Gamma^{\op})$ for any oriented graph $\Gamma$. 
		Therefore, $H$ can not be isomorphic to the oriented pro-$\ell$ RAAG associated to an oriented graph $\Gamma$ and a linear orientation.
	\end{proof}

	
	
	\section{Oriented pro-$\ell$ RAAGs and maximal pro-$\ell$ Galois groups}
	
	\subsection{Oriented pro-$\ell$ RAAGs of elementary type}
	
	The following fact is rather straightforward, and it follows from the inductive procedure to construct oriented graphs of elementary type and oriented pro-$\ell$ groups of elementary type (cf. \S~\ref{ssec:ETC}).
	
	\begin{fact}\label{prop:RAAGsET}
		Let $\Gamma=(\euV,\euE)$ be an oriented graph, and let $\lambda$ be a non trivial linear orientation.
		Then $(G_{\Gamma,\lambda},\theta_{\Gamma,\lambda})$ is an oriented pro-$\ell$ group
		of elementary type if, and only if, $\Gamma$ is of elementary type.
	\end{fact}
	
	Moreover, from Proposition~\ref{prop:subgps etc} one deduces the following.
	
	\begin{fact}\label{fact:subgroups raags et}
		Let $\Gamma=(\euV,\euE)$ be an oriented graph, and let $\lambda$ be a non trivial linear orientation.
		If $\Gamma$ is of elementary type, then for every subgroup $H$ of $G_{\Gamma,\lambda}$ one has
		\[ (H,\theta_{\Gamma,\lambda}\vert_H)\simeq (G_{\Gamma',\lambda'},\theta_{\Gamma',\lambda'}) \]
		for some oriented graph $\Gamma'$ of elementary type and some non trivial linear orientation $\lambda'$.
	\end{fact}

	The following is the group-theoretic analogue of Proposition~\ref{prop:graphsET}.
	
	\begin{prop}\label{prop:badgraphs badgroups}
		Let $\Gamma=(\euV,\euE)$ be a specially oriented graph, and let $\lambda$ be a linear orientation.
		Then $\Gamma$ is of elementary type if, and only if,
		$G_{\Gamma,\lambda}$ has no subgroups isomorphic to $G_{\Gamma',\lambda'}$, with $\Gamma'\in\{\Lambda_{\mathrm{s}},\mathrm{C}_4,\mathrm{L}_3\}$ and $\lambda'$ a linear orientation.
	\end{prop}
	
	In the statement of Proposition~\ref{prop:badgraphs badgroups}, $\mathrm{C}_4$ and $\mathrm{L}_3$ denote --- with an abuse of notation --- those specially oriented graphs with no special vertices and edges whose associated na\"ive graphs are equal to $\mathrm{C}_4$ and $\mathrm{L}_3$ (as defined in Example~\ref{ex:C4L3}) respectively. 
	In order to prove Proposition~\ref{prop:badgraphs badgroups}, we need the following technical lemma.
	
	\begin{lem}\label{prop:Lambda2}
		Let $\Gamma=(\euV,\euE)$ be a specially oriented graph and let $\lambda$ be a linear orientation.	If $\Gamma$ has an induced subgraph $\Gamma'$ such that $\ddot\Gamma'$ has geometric realization
		\[ \xymatrix@R=1.5pt{v_1 & x & v_2 \\ \bullet\ar@{-}[r] &\bullet& \bullet\ar@{-}[l]  } \]
		then the subgroup $V$ of $G$ generated by $\{v_1,v_2\}$ is a free pro-$\ell$ group.
	\end{lem}
	
	\begin{proof}
		Let $H,H_1,H_2$ be the subgroups of $G$ generated respectively by $\{x,v_1,v_2\}$, $\{x,v_1\}$, and $\{x,v_2\}$.
		By Proposition~\ref{prop:specialRAAG inclusion}, $H_1,H_2$ are locally uniform --- in particular, for $i=1,2$ one has 
		$H_i=\langle\: v_i\:\rangle\rtimes \langle\:x\:\rangle$, with ${}^xv_i=v_i^{\lambda(1)}$, if $x\in\euV_\mathrm{s}$; while 
		$H_i=\langle\: x\:\rangle\rtimes \langle\:v_i\:\rangle$, with ${}^{v_i}x=x^{\theta_{\Gamma,\lambda}(v_i)}$, if $x\in\euV_\mathrm{o}$.
		
		Set $Z=\langle\: x\:\rangle$ and $V_i=\langle\:v_i\:\rangle$, with $i=1,2$, and let $\mathrm{T}=(\euV(\mathrm{T}),\euE(\mathrm{T}))$ be the second-countable pro-$\ell$ tree with 
		\[
		\euV(\mathrm{T})=\{\:hH_1,hH_2\:\mid\:h\in H\:\},\qquad\euE(\mathrm{T})=\{\:hZ\:\mid\:h\in H\:\}.
		\]
		Then $H$ acts naturally on $\mathrm{T}$ by $g\cdot(hH_i)=(gh)H_i$ and $g\cdot(hZ)=(gh)Z$.
		The stabilizers in $V$ of and edge $hZ$ and a vertex $hH_i$ of $\mathrm{T}$ are respectively
		\[
		\begin{split}
			\mathrm{Stab}_V(hZ)&=V\cap {}^hZ=V\cap\left\langle\:{}^hx\:\right\rangle,\\
			\mathrm{Stab}_V(hH_i)&=V\cap {}^hH_i=V\cap\left\langle\:{}^hv_i,{}^hx\:\right\rangle.
		\end{split}
		\]
		We have two cases.
		If $x\in\euV_\mathrm{s}$, then ${}^{v_i}x=v_i^{1-\lambda(1)}x$ for both $i=1,2$, and we may write $h=vx^\mu$, for some $v\in V$ and $\mu\in\Z_{\ell}$, as $H=V\rtimes Z$.
		Then 
		\[ {}^hx={}^vx=v'\cdot x,\qquad {}^hv_i=vv_i^{\lambda(1)^\mu}v^{-1},\]
		for some $v'\in\Phi(V)$.
		If $x\in\euV_\mathrm{o}$, then $H=Z\rtimes V$, and we may write $h=x^\mu v$ for some $v\in V$ and $\mu\in\Z_{\ell}$.
		Then 
		\[
		{}^hx= x^{\theta_{\Gamma,\lambda}(v)},\qquad {}^hv_i=x^{\mu(1-\theta_\Gamma(v_i))}\cdot {}^vv_i.
		\]
		In both cases, $\mathrm{Stab}_V(hZ)=\{1\}$ and $\mathrm{Stab}_V(hH_i)=\langle\:vv_iv^{-1}\:\rangle\simeq\Z_{\ell}$.
		By \cite[Thm.~5.6]{melnikov:freeprod}, $V$ is isomorphic to the free pro-$\ell$ product of some $\mathrm{Stab}_V(hH_i)$ and of a free pro-$\ell$ group, and therefore $V$ is a free pro-$\ell$ group.
	\end{proof}

	\begin{proof}[Proof of Proposition~\ref{prop:badgraphs badgroups}]
		Set $G=G_{\Gamma,\lambda}$.
		If $\Gamma$ is of elementary type, then $(G,\theta_{\Gamma,\lambda})$ is of elementary type by Fact~\ref{prop:RAAGsET}.
		Hence, for every finitely generated subgroup $H\subseteq G$ the oriented pro-$\ell$ group $(H,\theta_{\Gamma,\lambda}\vert_H)$ is isomorphic to an oriented pro-$\ell$ RAAG associated to some oriented graph of elementary type by Fact~\ref{fact:subgroups raags et}.
		
		Conversely, suppose that $\Gamma$ is not of elementary type.
		By Proposition~\ref{prop:graphsET}, $\Gamma$ has an induced subgraph $\Gamma'=(\euV',\euE')$ such that either $\Gamma'=\Lambda_{\mathrm{s}}$, or $\ddot\Gamma'\in\{\mathrm{C}_4,\mathrm{L}_3\}$.
		
		Assume first that $\Gamma'=\Lambda_{\mathrm{s}}$, with $\euV'=\{x,v_1,v_2\}$.
		By Proposition~\ref{prop:specialRAAG inclusion}, the subgroup of $G$ generated by $x$ is isomorphic to $\Z_{\ell}$, while by Lemma~\ref{prop:Lambda2}, the subgroup of $G$ generated by $\{v_1,v_2\}$ is a free pro-$\ell$ group.
		Hence, the subgroup of $G$ generated by $\euV'$ is 
		$$\left\langle\: x,v_1,v_2\:\mid\: {}^xv_i=v_i^{\lambda(1)},\:i=1,2\:\right\rangle\simeq G_{\Lambda_{\mathrm{s}},\lambda}.$$
		
		Assume now that $\ddot\Gamma'=\mathrm{C}_4$, with 
		$$\euV'=\{\:v_1,\:v_2,\:v_3,\:v_4\:\}\qquad\text{and}\qquad\ddot\euE'=\left\{\:\{v_1,v_2\},\:\ldots,\:\{v_4,v_1\}\:\right\}.$$
		By Lemma~\ref{prop:Lambda2}, the two subgroups $H_1,H_2$ of $G$ generated by $v_1,v_3$ and by $v_2,v_4$ respectively, are 2-generated free pro-$\ell$ groups.
		Now pick $y_1,y_3\in\kernel(\theta_{\Gamma,\lambda}\vert_{H_1})$ and $y_2,y_4\in\kernel(\theta_{\Gamma,\lambda}\vert_{H_2})$ such that the subgroup of $H_1$ generated by $y_1,y_3$ is not isomorphic to $\Z_\ell$, and analogously $y_2,y_4$ --- this is possible as both kernels are non-abelian free pro-$\ell$ groups.
		Then the subgroup of $G$ generated by $y_1,y_2,y_3,y_4$ is
		\[
		\left\langle\: y_1,y_2,y_3,y_4\:\mid\: [y_i,y_j]=1,\:i=1,3,\:j=2,4\:\right\rangle\simeq G_{{\mathrm{C}_4},\lambda'}
		\]
		with $\lambda'$ an arbitrary linear orientation.

		Finally, assume that $\ddot\Gamma'=\mathrm{L}_3$, with
		$$\euV'=\{\:v_1,\:v_2,\:v_3,\:v_4\:\}\qquad\text{and}\qquad\ddot\euE'=\left\{\:\{v_1,v_2\},\:\{v_2,v_3\},\:\{v_3,v_4\}\:\right\}.$$
		If $v_2\in\euV_\mathrm{s}$, then $(v_2,v_1),(v_2,v_3)\in\euE_\mathrm{s}$, as $\Gamma$ is specially oriented, and thus the induced subgraph of $\Gamma$ with vertices $v_1,v_2,v_3$ is $\Lambda_{\mathrm{s}}$: in this case $G$ contains a subgroup isomorphic to $G_{\Lambda_{\mathrm{s}},\lambda'}$ for some $\lambda'$ by the above argument.
		Analogously if $v_3\in\euV_\mathrm{s}$.
		Otherwise, suppose that $v_2,v_3\in\euV_\mathrm{o}$, so that $(v_2,v_3)\in\euE_\mathrm{o}$.
		By Lemma~\ref{prop:Lambda2}, the two subgroups $H_1,H_2$ of $G$ generated by $v_1,v_3$ and by $v_2,v_4$ respectively, are 2-generated free pro-$\ell$ groups.
		Put $y_1=[v_1,v_2]$ and $y_4=[v_3,v_4]$ --- observe that $y_1,y_4\in\kernel(\theta_{\Gamma,\lambda})$.
		Then the subgroup of $G$ generated by $y_1,v_2,v_3,y_4$ is
		\[
		\left\langle\: y_1,v_2,v_3,y_4\:\mid\: [y_1,v_2]=[v_2,v_3]=[v_3,y_4]=1\:\right\rangle\simeq G_{{\mathrm{L}_3},\lambda'}
		\]
		with $\lambda'$ an arbitrary linear orientation.
	\end{proof}

	
	\subsection{1-cyclotomic oriented pro-$\ell$ RAAGs}
	
	By Theorem~\ref{thm:cyc pRAAG}, an oriented pro-$\ell$ RAAG associated to an oriented graph $\Gamma$ may be 1-cyclotomic only if $\Gamma$ is specially oriented.
	We show that, in fact, 1-cyclotomicity is far more restrictive.
	
	\begin{prop}\label{prop:Lambdas nocyc}
		Let $\lambda$ be a linear orientation.
		Then the oriented pro-$\ell$ RAAG $(G_{\Lambda_{\mathrm{s}},\lambda},\theta_{\Lambda_{\mathrm{s}},\lambda})$ is not 1-cyclotomic.
	\end{prop}
	
	\begin{proof}
		We keep the same notation as in the proof of Proposition~\ref{thm:lambdas}.
		Recall that $G_{\Lambda_{\mathrm{s}},\lambda}$ has a subgroup $H$ with minimal presentation
		\[ H=\left\langle\:u_1,y,t\:\mid\:{}^yu_1=u_1^{1+q},\: \left[u_1y,[t^{-1},y]\right]=u_2^{-q}\left({}^{(u_2^{-1})}u_1\right)^q\:\right\rangle,\]
		where $u_1=z_1^q$, $y=u_1^{-1}x$, $t=z_1z_2^{-1}$, cf. \eqref{eq:pres H BK}.
		
		Let $\phi_V\colon H\to\Z/{\ell}\Z$ be the homomorphism defined by $\phi_V(y)=1$, $\phi(u_1)=\phi(t)=0$, and set $v=y^{\ell}$,  $w=[y,t^{-1}]$.
		Let $\phi_V\colon H\to\Z/{\ell}\Z$ be the homomorphism defined by $\phi_V(y)=1$, $\phi(u_1)=\phi(t)=0$.
		Then $V$ is a normal subgroup of $H$ of index $\ell$, and it is generated, as a normal subgroup, by $v,u_1,t^{-1}$.
		In fact, $V$ is generated as a pro-$\ell$ group by the set
		\[
		\left\{\:v,\:t^{-1},\:[y,t^{-1}]=w,\:\left[y,[y,t^{-1}]\right],\:\ldots,\:
		[\underbrace{y,[\ldots,[y}_{p-1\text{ times}},t^{-1}]\ldots]],\:u_1\:\right\}, \]
		as $H/V=\{V,yV,\ldots,y^{p-1}V\}$, and ${}^yu_1=u_1^{1+q}$.
		Observe that by \eqref{eq: u2 gen by}, $u_2$ lies in the subgroup of $V$ generated by $t,w,u_1$, and thus the relation
		\[ {}^{u_1}\left[y,[y,t^{-1}]\right][u_1,w] = \left[u_1y,[y,t^{-1}]\right]
		= u_1^{-q}\left({}^{u_1^{-1}}u_2\right)^q\in\langle\:t,w,u_1\:\rangle^q
		\]
		implies that every higher commutator $[y,[\ldots,[y,t^{-1}]\ldots]]$ --- and thus the whole pro-$\ell$ group $V$ --- is minimally generated by the set $\{v,t^{-1},w,u_1\}$.
		
		Set $\theta=\theta_{\Lambda_{\mathrm{s}},\lambda}\vert_V\colon V\to\Z_\ell^\times$.
		Since $\theta(v)=(1+q)^{\ell}$ and $\theta(u_1)=\theta(t)=\theta(w)=1$, $u_1$ lies in $\kernel(\theta)$.
		Now let $N$ be the normal subgroup of $V$ generated by $u_1$ as a normal subgroup.
		Then the map $N/N^\ell[N,V]\to V/\Phi(V)$ is injective, and hence Remark~\ref{rem:kummer quotient} implies that if the oriented pro-$\ell$ group $(V,\theta)$ is Kummerian, then also the oriented pro-$\ell$ group $(V/N,\theta_{/N})$ is Kummerian.
		We claim that the latter is not Kummerian.
		Indeed, $\{vN,tN,wN\}$ is a minimal generating set of $V/N$.
		Moreover, one has $u_2\equiv w\bmod N$, and thus 
		\[
		{}^yw={}^y[y,t^{-1}] \equiv w^{1+q}\mod N.	\]
		Therefore, one computes
		\[ \begin{split}
			\left[v,t^{-1}\right]=\left[y^{\ell},t^{-1}\right]&=
			{}^{y^{{\ell}-1}}[y,t^{-1}]\cdot{}^{y^{{\ell}-2}}[y,t^{-1}]\cdots{}^y[y,t^{-1}]\cdot [y,t^{-1}]\\ 
			&\equiv w^{(1+q)^{{\ell}-1}}\cdot w^{(1+q)^{{\ell}-2}}\cdots w^{1+q}\cdot w\mod N\\
			&\equiv w^{1+(1+q)+\ldots+(1+q)^{{\ell}-2}+(1+q)^{{\ell}-1}}\mod N.
		\end{split}\]
		This yields a relation $[vN,tN]=(wN)^\mu$, with $\mu\in {\ell}\Z_{\ell}$, $\mu\neq0$.
		Hence, by \cite[Thm.~8.1]{eq:kummer}, the oriented pro-$\ell$ group $(V/N,\theta_{/N})$ is not Kummerian.
		Thus, $(G,\theta_{\Lambda_{\mathrm{s}},\lambda})$ is not a 1-cyclotomic oriented pro-$\ell$ group.
	\end{proof}
	
	
	\subsection{Proof of Theorem~\ref{thm:main}}
	
	\begin{proof}[Proof of Theorem~\ref{thm:main}]
		Let $\Gamma=(\euV,\euE)$ be an oriented graph, let $\lambda\colon \Z_\ell\to\Z_\ell^\times$ be a linear orientation, and set $G=G_{\Gamma,\lambda}$ and $\theta=\theta_{\Gamma,\lambda}\colon G\to\Z_\ell$.
		
		The equivalence between (0) and (iv) is stated in Fact~\ref{prop:RAAGsET}.
		
		It is well-known that (iv) implies (i), as the realizability as a maximal pro-$\ell$ groups is closed with respect to free pro-$\ell$ products and semi-direct products with $\Z_\ell$ (cf. ...). 
		Moreover, (iv) implies also (ii) and (iii), cf. \cite[Thm.~1.4]{qw:cyclotomic} (see also Proposition~\ref{prop:subgps etc}).
		Finally, (iv) implies also (v), as stated in Fact~\ref{fact:subgroups raags et}.
		
		The Norm Residue Theorem (cf. \cite{HW:book}) yields the implications (i) $\Rightarrow$ (ii) (cf. \cite[\S~2]{cq:bk}) and (i) $\Rightarrow$ (iii) (cf. \cite[Thm.~1.1]{qw:cyclotomic}).
		
		Assume that $G$ is Bloch-Kato.
		Then $\Gamma$ is specially oriented by Proposition~\ref{thm:nospecial noBK}.
		Suppose that $G$ contains a subgroup $H$ isomorphic to $G_{\Gamma',\lambda}$ with either $\Gamma'\in\{\Lambda_{\mathrm{s}},\mathrm{C}_4,\mathrm{L}_3\}$.
		If $\Gamma'=\Lambda_\mathrm{s}$, then $H$ is not Bloch-Kato by Proposition~\ref{thm:lambdas}.
		If $\Gamma'=\mathrm{C}_4,\mathrm{L}_3$, then $H$ is not Bloch-Kato by \cite[Thm.~1.2]{SZ:RAAGs}.
		Therefore, $G$ contains no subgroups isomorphic to $G_{\Gamma',\lambda}$ for such an oriented graph $\Gamma'$, and $(G,\theta_{\Gamma,\lambda})$ is of elementary type by Proposition~\ref{prop:badgraphs badgroups}.
		This proves the implication (iii) $\Rightarrow$ (iv).
		
		Assume $(G,\theta)$ is 1-cyclotomic. 
		Then by Theorem~\ref{thm:cyc pRAAG}, $\Gamma$ is specially oriented, and $\theta=\theta_{\Gamma,\lambda}$.
		Suppose that $G$ contains a subgroup $H$ isomorphic to $G_{\Gamma',\lambda}$ with either $\Gamma'\in\{\Lambda_{\mathrm{s}}\mathrm{C}_4,\mathrm{L}_3\}$.
		If $\Gamma'=\Lambda_\mathrm{s}$, then $H$ can not be completed into a 1-cyclotomic oriented pro-$\ell$ group by Proposition~\ref{prop:Lambdas nocyc}.
		If $\Gamma'=\mathrm{C}_4,\mathrm{L}_3$, then $H$ can not be completed into a 1-cyclotomic oriented pro-$\ell$ group by \cite[Thm.~1.5]{SZ:RAAGs} and Theorem~\ref{thm:cyc pRAAG}.
		Therefore, $G$ contains no subgroups isomorphic to $G_{\Gamma',\lambda}$ for such an oriented graph $\Gamma'$, and $(G,\theta_{\Gamma,\lambda})$ is of elementary type by Proposition~\ref{prop:badgraphs badgroups}.
		This proves the implication (ii) $\Rightarrow$ (iv).
		
		Finally, assume that every finitely generated subgroup of $G$ is isomorphic to an oriented pro-$\ell$ RAAG.
		Then $\Gamma$ is specially oriented by Corollary~\ref{cor:nospecial nosubgroups}.
		Suppose that $G$ contains a subgroup $H$ isomorphic to $G_{\Gamma',\lambda}$ with either $\Gamma'\in\{\Lambda_{\mathrm{s}},\mathrm{C}_4,\mathrm{L}_3\}$.
		If $\Gamma'=\Lambda_\mathrm{s}$, then $H$ contains a subgroup which can not occur as an oriented pro-$\ell$ RAAG by Corollary~\ref{cor:lambdasRAAG}.
		If $\Gamma'=\mathrm{C}_4,\mathrm{L}_3$, then $H$ contains a subgroup which can not occur as an oriented pro-$\ell$ RAAG by \cite[Thm.~1.2]{SZ:RAAGs}.
		Therefore, $G$ contains no subgroups isomorphic to $G_{\Gamma',\lambda}$ for such an oriented graph $\Gamma'$, and $(G,\theta_{\Gamma,\lambda})$ is of elementary type by Proposition~\ref{prop:badgraphs badgroups}.
		This proves the implication (iv) $\Rightarrow$ (iv).
	\end{proof}

	\section{Chordal oriented graphs}\label{sec:chord}


	\subsection{Chordal graphs and patching of graphs}\label{ssec:chord}
	
	\begin{defin}\label{defi:chordalgraph}\rm
		A na\"ive graph $\ddot\Gamma=(\ddot\euV,\ddot\euE)$ is said to be {\sl chordal} (or {\sl triangulated}) if there are no induced subgraphs of $\Gamma$ which are circuits of length at least 4. 
		An oriented graph $\Gamma=(\euV,\euE)$ is said to be chordal if $\ddot\Gamma$ is a chordal na\"ive graph.
	\end{defin}
	
	Chordal graphs have the following characterization (cf. \cite[Prop.~5.5.1]{graphbook} and \cite[Thm.~3.2]{chordalgraphs}).
	
	\begin{prop}\label{prop:chordalgraphs}
		Let $\ddot\Gamma=(\ddot\euV,\ddot\euE)$ be a na\"ive graph.
		Then the following are equivalent.
		\begin{itemize}
			\item[(i)] The graph $\ddot\Gamma$ is chordal.
			\item[(ii)] The graph $\ddot\Gamma$ decomposes as patching of two induced proper subgraphs
			$\ddot\Gamma_1,\ddot\Gamma_2$ which are chordal, along a common clique $\ddot\Delta\subseteq\ddot\Gamma_1,\ddot\Gamma_2$.
			\item[(iii)] The graph $\Upsilon(\ddot\Gamma)$ has a maximal subtree $\mathrm{T}_{\Upsilon(\ddot\Gamma)}$ with the clique-intersection property.
		\end{itemize}
	\end{prop}
	
	\begin{exam}\label{exam:chord1}\rm
		Consider the na\"ive graph $\ddot\Gamma=(\ddot\euV,\ddot\euE)$ with geometric realization
		
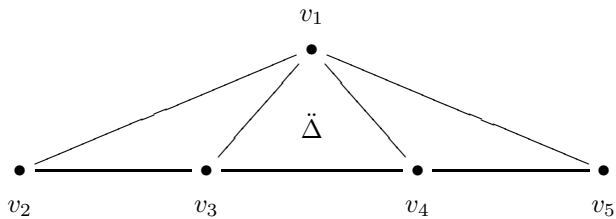
\[ 
		\xymatrix@R=1.5pt{  &&& v_1 &&& \\ &&& \bullet\ar@{-}[ddddlll]\ar@{-}[ldddd]\ar@{-}[rdddd]\ar@{-}[rrrdddd] & && \\
			\\ \\ &&&\ddot\Delta&&&  \\ 
			\bullet\ar@{-}[rr] && \bullet\ar@{-}[rr] && \bullet\ar@{-}[rr] && \bullet
			\\  v_2 && v_3 && v_4 && v_5}
		\]
		Then $\ddot\Gamma$ is chordal, and it is the pasting of the two induced subgraphs $\ddot\Gamma_1=(\ddot\euV_1,\ddot\euE_1)$ and $\ddot\Gamma_2=(\ddot\euV_2,\ddot\euE_2)$, with $\ddot\euV_1=\ddot\euV\smallsetminus\{v_5\}$ and $\ddot\euV_2=\euV\smallsetminus\{v_2\}$, along the common subgraph $\ddot\Delta$, which is the triangle with vertices $v_1,v_3,v_4$.
		Moreover, if $\ddot\Delta'$ and $\ddot\Delta''$ are the triangles with vertices $v_1,v_2,v_3$ and $v_1,v_4,v_5$ respectively, then $\ddot\Gamma_1$ may be obtained as the pasting of $\ddot\Delta'$ and $\ddot\Delta$ along the common edge with vertices $v_1,v_3$, and analogously $\ddot\Gamma_2$ is the pasting of $\ddot\Delta$ and $\ddot\Delta''$ along the common edge with vertices $v_1,v_4$.
	\end{exam}
	
	\subsection{Chordal oriented graphs and oriented pro-$\ell$ RAAGs}
	
	Let $\Gamma=(\euV,\euE)$ be a specially oriented graph, and let $\lambda$ be a linear orientation.
	Recall that by Proposition~\ref{prop:specialRAAG inclusion}, if $\Delta=(\euV(\Delta),\euE(\Delta))$ is a clique of $\Gamma$, then the inclusion $\euV(\Delta)\hookrightarrow\euV$ induces a monomorphism of pro-$\ell$ groups $G_{\Delta,\lambda}\to G_{\Gamma,\lambda}$.
	Hence, if $\Gamma$ is chordal, then one may find two proper induced subgraphs $\Gamma_1,\Gamma_2$ of $\Gamma$, whose intersection is a clique $\Delta$, such that $\Gamma$ is the patching of $\Gamma_1,\Gamma_2$ along $\Delta$, so that 
	\begin{equation}\label{eq:amalg 1}
		(G_{\Gamma,\lambda},\theta_{\Gamma,\lambda})\simeq 
		(G_{\Gamma_1,\lambda},\theta_{\Gamma_1,\lambda})\amalg_{G_{\Delta,\lambda}}^{\hat\ell}(G_{\Gamma_2,\lambda},\theta_{\Gamma_2,\lambda})
	\end{equation}
	(see \ref{rem:RAAG freeprod}--(ii)).
	Moreover, by \cite[Prop.~5.22]{qsv:quadratic} this amalgamated free pro-$\ell$ product is {\sl proper} --- i.e., the two factors are subgroups of the amalgamated free pro-$\ell$ product ---, as $G_{\Delta,\lambda}$ is a locally uniform pro-$\ell$ group.
	Therefore, an oriented pro-$\ell$ RAAG associated to a chordal specially oriented graph may be constructed iterating amalgamated free pro-$\ell$ products over locally uniform subgroups, starting from oriented pro-$\ell$ RAAGs associated to complete specially oriented graphs.
	
	In particular, let $\mathrm{T}=\mathrm{T}_{\Upsilon(\Gamma)}$ be a maximal subtree of the clique-graph $\Upsilon(\Gamma)$ with the clique-intersection property.
	Then one has a decomposition as proper amalgamated free pro-$\ell$ product
	\[
	G_{\Gamma,\lambda}\simeq
	\dfrac{\coprod_{\Delta\in\mathrm{\bf mx}(\Gamma)}^{\hat {\ell}}G_{\Delta,\lambda}}{N},
	\]
	where $N$ is the normal subgroup of the free pro-$\ell$ product $\coprod_{\Delta}G_{\Delta,\lambda}$ generated by the elements
	$$\iota_{\Xi,\Delta}(v)\cdot\iota_{\Xi,\Delta'}(v)^{-1},\qquad  
	v\in\euV(\Delta),\:\Xi=\Delta\cap\Delta',\: (\Delta,\Delta')\in\euE(\mathrm{T}),$$
	and $\iota_{\Xi,\Delta}\colon G_{\Xi,\lambda}\to G_{\Delta,\lambda}$ is the monomorphism of locally uniform pro-$\ell$ groups induced by $\euV(\Xi)\hookrightarrow\euV(\Delta)$ (cf. \cite{simoblu}).
	
	\begin{exam}\rm
		Let $\lambda$ be a linear orientation, and let
		$\Gamma=(\euV,\euE)$ be a specially oriented graph with associated na\"ive graph $\ddot\Gamma$ as the chordal graph in Example~\ref{exam:chord1}. 
		Moreover, let $\Delta$ be the clique of $\Gamma$ with vertices $v_1,v_3,v_4$, and analogously $\Delta'$ and $\Delta''$.
		Then the clique-graph $\Upsilon(\Gamma)$ has geometric realization
		\[\xymatrix@R=1.5pt{  \Delta' & \Delta & \Delta'' \\ \bullet\ar@{-}[r]_-{\Xi_1} & \bullet\ar@{-}[r]_-{\Xi_2} &\bullet
		}
		\]
		where $\Xi_1=\Delta'\cap\Delta$ and $\Xi_2=\Delta\cap\Delta''$ are the 2-cliques with vertices $v_1,v_3$ and $v_1,v_4$ respectively.
		Therefore,
		\begin{equation}\label{eq:amalg simone}
			G_{\Gamma,\lambda}\simeq \left(G_{\Delta'}\amalg_{G_{\Xi_1,\lambda}}^{\hat\ell}G_{\Delta,\lambda}\right)
			\amalg_{G_{\Xi_2,\lambda}}^{\hat\ell}G_{\Delta'',\lambda},
		\end{equation}
		Observe that $G_{\Delta',\lambda},G_{\Delta,\lambda},G_{\Delta'',\lambda}$ are 3-generated locally uniform pro-$\ell$ groups, while $G_{\Xi_1,\lambda},G_{\Xi_2,\lambda}$ are 2-generated locally uniform pro-$\ell$ groups.
	\end{exam}
	
	\subsection{Chordal oriented graphs and cohomology}
	
	In \cite[Thm.~H]{qsv:quadratic} it is shown that a generalized pro-$\ell$ RAAG associated to an oriented graph $\Gamma$ satisfying a particular group-theoretic condition has $\F_\ell$-cohomology isomorphic to the exterior Stanley-Reisner algebra $\bfLam_\bullet(\ddot\Gamma^{\op})$.
	Thus, the following theorem is a refinement of the aforementioned result.
	
	\begin{thm}\label{thm:chordal cohomology}
		Let $\Gamma=(\euV,\euE)$ be an oriented graph, and let $\lambda$ be a linear orientation.
		If $\Gamma$ is chordal and specially oriented, then 
		\[
		\bfH^\bullet(G_{\Gamma,\lambda},\F_\ell)\simeq\bfLam_\bullet(\ddot\Gamma^{\op}).
		\] 
	\end{thm}
	
	In order to prove Theorem~\ref{thm:chordal cohomology}, we need a technical lemma, which is a slight modification of \cite[Prop.~5.21]{qsv:quadratic}.
	Given a pro-$\ell$ group $G$, for $n\geq 1$ let $\Phi^n(G)$ denote the $n$-th term of the Frattini series of $G$ --- namely, $\Phi^1(G)=G$ and $\Phi^{n+1}(G)=\Phi(\Phi^n(G))$.
	If $G$ is locally uniform, then $\Phi^n(G)=G^{\ell^n}$ (cf. \cite[Thm.~3.6]{ddsms}).
	
	Recall that, given a specially oriented graph $\Gamma=(\euV,\euE)$ containing a clique $\Delta=(\euV(\Delta),\euE(\Delta))$, the inclusion $\euV(\Delta)\hookrightarrow\euV$ induces a monomorphism of pro-$\ell$ groups $\phi_\Delta\colon G_{\Delta,\lambda}\to G_{\Gamma,\lambda}$ (cf. Proposition~\ref{prop:specialRAAG inclusion}), with $\lambda$ a linear orientation.
	
	\begin{lem}\label{lem:amalg}
		Let $\Gamma=(\euV,\euE)$ be a specially oriented graph, and let $\lambda$ be a linear orientation.
		If $\Delta=(\euV(\Delta),\euE(\Delta))$ is a clique of $\Gamma$, then 
		\[
		\Phi^n(G_{\Delta,\lambda})=G_{\Delta,\lambda}\cap\Phi^n(G_{\Gamma,\lambda})
		\]
		for all $n\geq1$ (where we consider $G_{\Delta,\lambda}$ as a subgroup of $G_{\Gamma,\lambda}$ via the monomorphism $\phi_\Delta$).
	\end{lem}
	
	\begin{proof}
		Set $G=G_{\Gamma,\lambda}$, $A=G_{\Delta,\lambda}$, and $\bar G=G/K_{\theta_{\Gamma,\lambda}}(G)$, and let $\varphi\colon G\to\bar G$ denote the canonical projection.
		Recall that $\varphi\vert_A$ is injective.
		Clearly, one has the inclusion $$\Phi^n(A)\subseteq A\cap\Phi^n(G).$$
		
		Let $F$ be the free pro-$\ell$ group generated by $\euV'$, where $\euV=\euV'\sqcup\euV(\Delta)$, and put
		$\tilde G=F\amalg^{\hat\ell} A$.
		Let $\pi\colon \tilde G\to G$ be the epimorphism which sends every vertex $v\in\euV'$, considered as an element of $F$, to the same vertex, considered as an element of $G$, and such that $\pi\vert_A=\phi_\Delta$.
		Altogether, one has a chain of epimorphisms of pro-$\ell$ groups
		\[
		\xymatrix{ \tilde G=F\amalg^{\hat\ell}A\ar[r]^-{\pi} & G\ar[r]^-{\varphi} & \bar G }.
		\]
		Pick an element $x\in A\cap\Phi^n(G)$.
		Then 
		\[
		\varphi(x)\in \varphi(A)\cap\Phi^n(\bar G)=\varphi(A)\cap \bar G^{\ell^n}=\varphi(A)^{\ell^n},
		\]
		as $\bar G$ is locally uniform, generated by $\varphi(\euV')\sqcup\varphi(\euV(\Delta))$.
		Since $\varphi\circ\phi_\Delta\colon A\to\varphi(A)$ is an isomorphism, there exists $y\in A^{\ell^n}$ such that $\varphi(\phi_\Delta(y))=\varphi(x)$, and hence $x=\phi_\Delta(y)$, namely, $x\in A^{\ell^n}$.
	\end{proof}

	\begin{proof}[Proof of Theorem~\ref{thm:chordal cohomology}]
		By \eqref{eq:H raag} one knows that $\rmH^n(G_{\Gamma,\lambda},\F_\ell)\simeq\Lambda_n(\ddot\Gamma^{\op})$ for $n=0,1,2$.
		Therefore, it sufficies to show that $\bfH^\bullet(G_{\Gamma,\lambda},\F_\ell)$ is a quadratic algebra.
		
		Let $\Gamma_1,\Gamma_2$ be proper induced subgraphs of $\Gamma$ whose intersection is a clique $\Delta$.
		We claim that the free amalgamated pro-$\ell$ product \eqref{eq:amalg 1} is proper.
		Indeed, for every $n\geq1$ set $U_n=\Phi^n(G_{\Gamma_1,\lambda})$ and $V_n=\Phi^n(G_{\Gamma_2,\lambda})$.
		Then $\{U_n\mid n\geq1\}$ and $\{V_n\mid n\geq1\}$ are basis of open neighbourhoods of 1 in $G_{\Gamma_1,\lambda}$ and $G_{\Gamma_2,\lambda}$ respectively.
		By Lemma~\ref{lem:amalg},
		$$U_n\cap G_{\Delta,\lambda}=V_n\cap G_{\Delta,\lambda}=G_{\Delta,\lambda}^{\ell^n}\qquad\text{for every }n\geq1,$$ and \cite[Thm.~9.2.4]{ribzal:book} implies that the amalgam is proper.
		
		Now, if $\Gamma$ is complete, then $G_{\Gamma,\lambda}$ is locally uniform, and thus $\bfH^\bullet(G_{\Gamma,\lambda},\F_\ell)\simeq\bfLam_\bullet(\ddot\Gamma^{\op})$ by Remark~\ref{rem:locunif cohom}.
		Otherwise, by the inductive procedure to construct chordal graphs we may assume that $\bfH^\bullet(G_{\Gamma_i,\lambda},\F_\ell)\simeq\bfLam_\bullet(\ddot\Gamma_i^{\op})$ for both $i=1,2$.
		Moreover, one has 
		\[
		\rmH^n(G_{\Gamma,\lambda},\F_\ell)\simeq \kernel\left(\Res_{G_{\Gamma_1,\lambda},G_{\Delta,\lambda}}^n\right)
		\oplus\rmH^n(G_{\Delta,\lambda},\F_\ell)\oplus 
		\kernel\left(\Res_{G_{\Gamma_2,\lambda},G_{\Delta,\lambda}}^n\right)
		\]
		for $n=1,2$.
		Hence, we may apply \cite[Thm.~B]{qsv:quadratic}, and $\bfH^\bullet(G_{\Gamma,\lambda},\F_\ell)$ is a quadratic algebra.
	\end{proof}
	
	\subsection{The Bogomolov-Positselski property}
	
	A Kummerian oriented pro-$\ell$ group $(G,\theta)$ with torsion-free orientation is said to have the {\sl Bogomolov-Positselski property} if $K_\theta(G)$ is a free pro-$\ell$ group (cf. \cite[\S~3]{qw:bogomolov}).
	From Theorem~\ref{thm:chordal cohomology} we deduce that an oriented pro-$\ell$ RAAG associated to a chordal specially oriented graph has the Bogomolov-Positselski property.
	
	\begin{thm}\label{thm:bogo}
		Let $\Gamma=(\euV,\euE)$ be a chordal specially oriented graph, and let $\lambda$ be a linear orientation.
		Then $(G_{\Gamma,\lambda},\theta_{\Gamma,\lambda})$ has the Bogomolov-Positselski property.
	\end{thm}
	
	\begin{proof}
		Since $\Gamma$ is specially oriented, $(G_{\Gamma,\lambda},\theta_{\Gamma,\lambda})$ is Kummerian by Theorem~\ref{thm:cyc pRAAG}.
		Let $\Gamma_1,\Gamma_2$ be proper induced subgraphs of $\Gamma$ whose intersection is a clique $\Delta$.
		By the proof of Theorem~\ref{thm:chordal cohomology}, the amalgamated free pro-$\ell$ product \eqref{eq:amalg 1} is proper, and moreover it satisfies the hypothesis of \cite[Thm.~5.5]{qw:bogomolov}.
		
		If $\Gamma$ is complete, then it has the Bogomolov-Positselski property as $K_{\theta_{\Gamma,\lambda}}(G_{\Gamma,\lambda})$ is trivial.
		Otherwise, by induction we may assume that $(G_{\Gamma_i,\lambda},\theta_{\Gamma_i,\lambda})$ has the Bogomolov-Positselski property for both $i=1,2$, and \cite[Thm.~5.5]{qw:bogomolov} yields the claim.
	\end{proof}
	
	Let $\K$ be a field containing a primitive $\ell$-th root of unity (and also $\sqrt{-1}$ if ${\ell}=2$).
	In \cite[Conj.~1.2]{pos:k}, L.~Positselski conjectures that $(G_{\K}(\ell),\hat\theta_{\K})$ has the Bogomolov-Positselski property.
	Therefore, Theorem~\ref{thm:bogo} implies Corollary~\ref{cor:main}--(iii).
	
	
	\subsection{Coherent oriented pro-$\ell$ RAAGs}\label{ssec:coherent}
	
	A finitely generated pro-$\ell$ group $G$ is said to be {\sl coherent} if every finitely generated subgroup $H\subseteq G$ is also finitely presented, i.e., if $|\rmH^1(H,\F_\ell)|<\infty$ implies $|\rmH^2(H,\F_\ell)|<\infty$.
	Moreover, $G$ is said to be {\sl of type $FP_\infty$} if $\rmH^n(G,\F_\ell)$ is finite for every $n\geq1$.
	We prove that every finitely generated subgroup of an oriented pro-$\ell$ RAAG associated to a chordal specially oriented graph --- even if it may not occur as an oriented pro-$\ell$ RAAG --- is of type $FP_\infty$.
	
	\begin{thm}
		Let $\Gamma=(\euV,\euE)$ be a chordal specially oriented graph, and let $\lambda$ be a linear orientation.
		Then every finitely generated subgroup of $G_{\Gamma,\lambda}$ is of type $FP_\infty$.
		In particular, $G_{\Gamma,\lambda}$ is coherent.
	\end{thm}
	
	\begin{proof}
		Set $G=G_{\Gamma,\lambda}$, and let $H$ be a finitely generated subgroup of $G$.
		Moreover, set 
		\[
		K=K_{\theta_{\Gamma,\lambda}}(G)\cap H\qquad\text{and}\qquad
		Q=\frac{HK_{\theta_{\Gamma,\lambda}}(G)}{K_{\theta_{\Gamma,\lambda}}(G)}.
		\]
		Then one has a short exact sequence of pro-$\ell$ groups 
		\[
		\xymatrix{  \{1\}\ar[r] & K \ar[r] & H \ar[r] & Q \ar[r] & \{1\}}
		\]
		where $K$ is a free pro-$\ell$ group by Theorem~\ref{thm:bogo}, and $Q$ is locally powerful as $\Gamma$ is specially oriented (cf. Theorem~\ref{thm:cyc pRAAG}).
		By \cite[Thm.~2, \S~3]{king}, $H$ is of type $FP_{\infty}$ if, and only if, $H_n(K,\Z_{\ell})$ is a finitely generated $\Z_{\ell}[\![Q]\!]$-module for each $n\geq1$.
		Since $H$ is finitely generated, $H_1(K,\Z_{\ell})$ is a finitely generated $\Z_{\ell}[\![Q]\!]$-module, while $H_n(K,\Z_{\ell})=0$ for each $n\geq2$ as $K$ is a free pro-$\ell$ group. 
		Hence, $H$ is of type $FP_\infty$.
		In particular, $\rmH^2(H,\F_{\ell})$ is finite, and thus $G$ is coherent.
	\end{proof}

\end{document}